\documentclass[a4paper]{article}
\usepackage{amsthm,amssymb,amsmath,enumerate}
\usepackage{algorithm}

\usepackage{cite}

\usepackage{tikz}
\usetikzlibrary{calc}
\usetikzlibrary{backgrounds}

\colorlet{hellgrau}{black!20!white}
\colorlet{dunkelgrau}{black!60!white}

\colorlet{grau}{black!40!white}
\colorlet{bold}{black} 
\tikzstyle{ledge}=[thick, grau]
\tikzstyle{rededge}=[very thick, bold]
\tikzstyle{lvertex}=[thick,circle,inner sep=0.cm, minimum size=2mm, fill=white, draw=grau]
\tikzstyle{redvx}=[thick,circle,inner sep=0.cm, minimum size=2mm, fill=white, draw=bold]

\tikzstyle{hvertex}=[thick,circle,inner sep=0.cm, minimum size=2mm, fill=white, draw=black]
\tikzstyle{hedge}=[very thick]
\tikzstyle{medge}=[thick]
\tikzstyle{harrow}=[thick,arrows=->]
\tikzstyle{darrow}=[thick,arrows=<-]
\tikzstyle{point}=[draw,circle,inner sep=0.cm, minimum size=1mm, fill=black]
\tikzstyle{pointer}=[thick,->,shorten >=2pt,color=dunkelgrau]
\tikzstyle{facebdry}=[color=auchblau, very thick] 
\tikzstyle{face}=[facebdry,fill=hellblau]
\tikzstyle{nface}=[color=hellblau,fill=hellblau,thick] 
\tikzset{>={latex}}
\tikzstyle{tinyvx}=[thick,circle,inner sep=0.cm, minimum size=1.3mm, fill=white, draw=black]
\tikzstyle{smallvx}=[hvertex,minimum size=1.7mm]

\pgfdeclarelayer{background}
\pgfdeclarelayer{foreground}
\pgfsetlayers{background, main, foreground}

\newtheorem*{lemma*linkageWalls}{Lemma~\ref{wall:lem:linkageMainLemma}}

\newcommand{\comment}[1]{}
\newcommand{\N}{\mathbb N}

\title{On the edge-Erd\H{o}s-P\'{o}sa property of walls}
\author{Henning Bruhn and Raphael Steck}
\date{\today}

\usepackage{hyperref}
\hypersetup{
pdftitle={On the edge-Erdös-P\'{o}sa property of walls}, 
pdfauthor={Raphael Steck},
pdfsubject={On the edge-Erdös-P\'{o}sa property of walls},
pdfproducer={pdfeTex 3.14159-1.30.6-2.2},
colorlinks=false,
pdfborder=0 0 0	
}

\usepackage{enumitem}

\theoremstyle{plain}
\newtheorem{theo}{Theorem}
\newtheorem{lemma}[theo]{Lemma}
\newtheorem{prop}[theo]{Proposition}
\newtheorem{defin}[theo]{Definition}

\begin{document}

\maketitle





\begin{abstract}
We show that walls of size at least $6 \times 4$ do not have the edge-Erd\H{o}s-P\'{o}sa property.
\end{abstract}

\section{Introduction}

The Erd\H{o}s-P\'{o}sa property provides a duality between packing and covering in graphs. We say that a class $\mathcal{F}$ has the \emph{edge-Erd\H{o}s-P\'{o}sa property} if there exists a function $f: \N \rightarrow \mathbb{R}$ such that for every graph $G$ and every integer $k$, there are $k$ edge-disjoint subgraphs of $G$ each isomorphic to some graph in $\mathcal{F}$ or there is an edge set $X \subseteq E(G)$ of size at most $f(k)$ meeting all subgraphs of $G$ isomorphic to some graph in $\mathcal{F}$.
In this article, we focus on graph classes that arise from taking minors. For a graph $H$, we define the set of \emph{$H$-expansions} as
\(
\mathcal{F}_H = \{ G \, | \, H \text{ is a minor of } G\}.
\)

While for all non-planar graphs $H$, $\mathcal{F}_H$ does not have the edge-Erd\H{o}s-P\'{o}sa property (see for example \cite{raymond17}), there are only some very simple planar graphs for which $\mathcal{F}_H$ is known to have the edge-Erd\H{o}s-P\'{o}sa property such as long cycles\cite{BHJ19} or $K_4$\cite{BH18}.
For most planar graphs, it is open whether they have the edge-Erd\H{o}s-P\'{o}sa property or not.

In this article, we show that 

\begin{theo} \label{wall:thm:WallsHaveNotEdgeErdosPosaIf6x4Contained}
For every wall $B$ of size at least $6 \times 4$, the class of $B$-expansions does not have the edge-Erd\H{o}s-P\'{o}sa property.
\end{theo}

\subsection{Condensed Wall}

The main gadget used for proving that walls do not have the edge-Erd\H{o}s-P\'{o}sa property is a wall-like structure called \emph{condensed wall} introduced by Bruhn et. al. \cite{bruhn18}, see Figure~\ref{fig:condWall}.

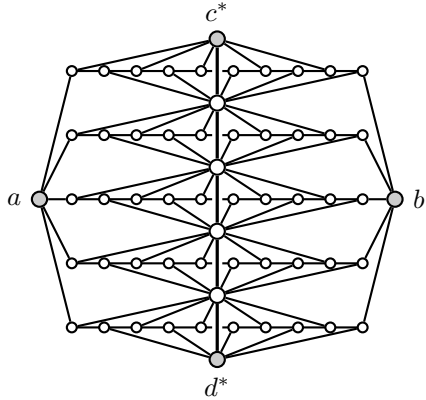
\begin{figure}[bht] 
\centering
\begin{tikzpicture}[scale=0.85]
\tikzstyle{tinyvx}=[thick,circle,inner sep=0.cm, minimum size=1.3mm, fill=white, draw=black]

\def\vstep{1}
\def\hstep{0.5}
\def\hwidth{9}
\def\hheight{4}

\def\totalheight{\hheight*\vstep}
\def\totalwidth{\hwidth*\hstep}
\pgfmathtruncatemacro{\minustwo}{\hwidth-2}
\pgfmathtruncatemacro{\minusone}{\hwidth-1}

\foreach \j in {0,...,\hheight} {
\draw[medge] (0,\j*\vstep) -- (\hwidth*\hstep,\j*\vstep);
\foreach \i in {0,...,\hwidth} {
\node[tinyvx] (v\i\j) at (\i*\hstep,\j*\vstep){};
}
}

\foreach \j in {1,...,\hheight}{
\node[hvertex] (z\j) at (0.5*\hwidth*\hstep,\j*\vstep-0.5*\vstep) {};
}
\pgfmathtruncatemacro{\plusvone}{\hheight+1}

\node[hvertex,fill=hellgrau,label=above:$c^*$] (z\plusvone) at (0.5*\totalwidth,\totalheight+0.5*\vstep) {};
\node[hvertex,fill=hellgrau,label=below:$d^*$] (z0) at (0.5*\totalwidth,-0.5*\vstep) {};

\foreach \j in {1,...,\plusvone}{
\pgfmathtruncatemacro{\subone}{\j-1}
\draw[line width=1.3pt,double distance=1.2pt,draw=white,double=black] (z\j) to (z\subone);
\foreach \i in {0,2,...,\hwidth}{
\draw[medge] (z\j) to (v\i\subone);
}
}

\foreach \j in {0,...,\hheight}{
\foreach \i in {1,3,...,\hwidth}{
\draw[medge] (z\j) to (v\i\j);
}
}

\pgfmathtruncatemacro{\minusvone}{\hheight-1}
\node[hvertex,fill=hellgrau,label=left:$a$] (a) at (-\hstep,0.5*\totalheight) {};
\foreach \j in {0,...,\hheight} {
\draw[medge] (a) -- (v0\j);
}

\node[hvertex,fill=hellgrau,label=right:$b$] (b) at (\totalwidth+\hstep,0.5*\totalheight) {};
\foreach \j in {0,...,\hheight} {
\draw[medge] (v\hwidth\j) to (b);
}
\end{tikzpicture}
\caption{A condensed wall of size 5.}
\label{fig:condWall}
\end{figure}

A condensed wall $W$ of size $r \in \N$ is the graph consisting of the following:
\begin{itemize}
\item For every $j \in [r]$, let $P^j = {u^j}_1, \ldots , {u^j}_{2r}$ be a path of length $2r - 1$ and for $j \in \{0\} \cup [r]$, let $z_j$ be a vertex. Moreover, let $a$, $b$ be two further vertices.
\item For every $i, j \in [r]$, add the edges $z_{j-1} {u^j}_{2i - 1}, z_{j} {u^j}_{2i}, z_{i-1} {z_i}, a {u^j}_{1}$ and $b {u^j}_{2r}$.
\end{itemize}
We define $c = z_0$ and $d = z_r$ and refer to
\begin{displaymath}
W_j = W[ \{ {u^j}_1, \ldots , {u^j}_{2r}, z_{j - 1}, z_j \} ]
\end{displaymath}

as the \emph{j-th layer of W}. 
Note that the layers of $W$ are precisely the blocks of $W - \{a, b\}$.
We will refer to the vertices $a, b$ quite often, and whenever we write $a$ or $b$ in this article, we refer to those vertices in a condensed wall.
The vertices connecting the layers of $W$ are $z_i, i \in \{0\} \cup [r]$, and we will call those \emph{bottleneck vertices}. This includes the vertices $c$ and $d$.
The edges $z_{i-1} z_i, i \in [r]$ are called \emph{jump-edges}.

For vertices $a, b, c, d$, an \emph{($a$-$b$, $c$-$d$)-linkage} is the vertex-disjoint union of an $a$-$b$-path with a $c$-$d$-path.

\begin{lemma}[Bruhn et. al. \cite{bruhn18}] \label{lem:noTwoLinkages}
There are no two edge-disjoint ($a$-$b$, $c$-$d$)-linkages in a condensed wall.
\end{lemma}

\subsection{Modification of the condensed wall}

A \emph{modified condensed wall} is a condensed wall without jump-edges $z_{i-1} z_i$ for all $i \in [r]$. 
Throughout, let $W$ be a condensed wall and let $W^-$ be a modified condensed wall that does not contain jump-edges $z_{i-1}z_i$.

\subsection{Definition of Walls}

For $m,n \in \N$, an \emph{elementary grid} of size $m \times n$ is a graph with vertices $v_{i,j}$ for all $i \in [m], j \in [n]$ and edges $v_{i,j} v_{i+1,j} \,\, \forall i \in [m-1], j \in [n]$ as well as $v_{i,j} v_{i,j+1} \,\, \forall i \in [m], j \in [n-1]$. A \emph{grid} is a subdivision of an elementary grid.

%
%
%
%
%

A wall is the subcubic variant of a grid. We define an \emph{elementary wall} as an elementary grid with every second vertical edge removed. That is, an elementary wall of size $m \times n$ is an elementary grid of size $m+1 \times 2n+2$ with every edge $v_{i,2j} v_{i+1,2j} \, , i \in [m], i \text{ is odd}, j \in [n+1]$ and every edge $v_{i,2j-1} v_{i+1,2j-1} \, , i \in [m], i \text{ is even}, j \in [n+1]$ being removed. Additionally, we remove all vertices of degree~$1$ and their incident edge.
The \emph{$i^\text{th}$ row} of an elementary wall is the induced subgraph on $v_{i,1},\ldots, v_{i,2n+2}$ for $i\in [m+1]$ (ignore the vertices that have been removed); this is a path. 
There are exactly $n+1$ disjoint paths between the first row and the $(m+1)^\text{th}$ row. These are the \emph{columns} of an elementary wall. The \emph{boundary} of an elementary wall is the union of the first and last row together with the first and last column. The \emph{bricks} of an elementary wall are its $6$-cycles. (See Figure~\ref{fig:elwall})

\begin{figure}[ht] 
\centering
\begin{tikzpicture}
\tikzstyle{tinyvx}=[thick,circle,inner sep=0.cm, minimum size=1.3mm, fill=white, draw=black]
\tikzstyle{vx}=[thick,circle,inner sep=0.cm, minimum size=1.6mm, fill=white, draw=black]
\tikzstyle{marked}=[line width=3pt,color=dunkelgrau]
\tikzstyle{point}=[thin,->,shorten >=2pt,color=dunkelgrau]
\tikzstyle{edg}=[draw,thick]

\def\wallheight{8}
\def\brickheight{0.4}

\pgfmathtruncatemacro{\lastrow}{\wallheight}
\pgfmathtruncatemacro{\penultimaterow}{\wallheight-1}
\pgfmathtruncatemacro{\lastrowshift}{mod(\wallheight,2)}
\pgfmathtruncatemacro{\lastx}{2*\wallheight+1}

\draw[edg] (\brickheight,0) -- (2*\wallheight*\brickheight+\brickheight,0);
\foreach \i in {1,...,\penultimaterow}{
  \draw[edg] (0,\i*\brickheight) -- (2*\wallheight*\brickheight+\brickheight,\i*\brickheight);
}
\draw[edg] (\lastrowshift*\brickheight,\lastrow*\brickheight) to ++(2*\wallheight*\brickheight,0);

\foreach \j in {0,2,...,\penultimaterow}{
  \foreach \i in {0,...,\wallheight}{
    \draw[edg] (2*\i*\brickheight+\brickheight,\j*\brickheight) to ++(0,\brickheight);
  }
}
\foreach \j in {1,3,...,\penultimaterow}{
  \foreach \i in {0,...,\wallheight}{
    \draw[edg] (2*\i*\brickheight,\j*\brickheight) to ++(0,\brickheight);
  }
}

\def\colind{5}
\foreach \j in {2,4,6}{
  \draw[marked] (\colind*\brickheight,\j*\brickheight-2*\brickheight) -- ++ (0,\brickheight) -- ++(-\brickheight,0) -- ++(0,\brickheight) -- ++(\brickheight,0);
}
\draw[marked] (\colind*\brickheight,6*\brickheight) -- ++ (0,\brickheight) -- ++(-\brickheight,0) -- ++(0,\brickheight);

\def\rowind{4}
\foreach \i in {1,...,\lastx}{
  \draw[marked] (\i*\brickheight-\brickheight,\rowind*\brickheight) -- ++(\brickheight,0);
}

\draw[marked] (2*\wallheight*\brickheight,1*\brickheight) -- ++(0,\brickheight) coordinate[midway] (brx)
-- ++(-2*\brickheight,0)
-- ++(0,-\brickheight) -- ++(2*\brickheight,0);

\foreach \i in {1,...,\lastx}{
  \node[tinyvx] (w\i w0) at (\i*\brickheight,0){};
}
\foreach \j in {1,...,\penultimaterow}{
  \foreach \i in {0,...,\lastx}{
    \node[tinyvx] (w\i w\j) at (\i*\brickheight,\j*\brickheight){};
  }
}
\foreach \i in {1,...,\lastx}{
  \node[tinyvx] (w\i w\lastrow) at (\i*\brickheight+\lastrowshift*\brickheight-\brickheight,\lastrow*\brickheight){};
}

\foreach \i in {2,4,...,\lastx}{
  \node[tinyvx,fill=white] (w\i w\lastrow) at (\i*\brickheight+\lastrowshift*\brickheight-\brickheight,\lastrow*\brickheight){};
}

\node[anchor=mid] (tr) at (\lastx*\brickheight+0.5,\wallheight*\brickheight+0.8){$1$\textsuperscript{st} row};
\draw[point,out=270,in=0] (tr) to (w\lastx w\wallheight);


\node[anchor=mid] (vp) at (0,\wallheight*\brickheight+0.8){column};
\draw[point,out=0,in=90] (vp) to (w\colind w\wallheight);

\node[align=center] (hp) at (\lastx*\brickheight+1.2,\rowind*\brickheight+0.8){row};
\draw[point,out=270,in=0] (hp) to (w\lastx w\rowind);

\node[align=center] (br) at (\lastx*\brickheight+1.2,1*\brickheight+0.8){brick};
\draw[point,out=270,in=0] (br) to (brx);

\end{tikzpicture}
\caption{An elementary wall of size $8\times8$.}\label{fig:elwall}
\end{figure}
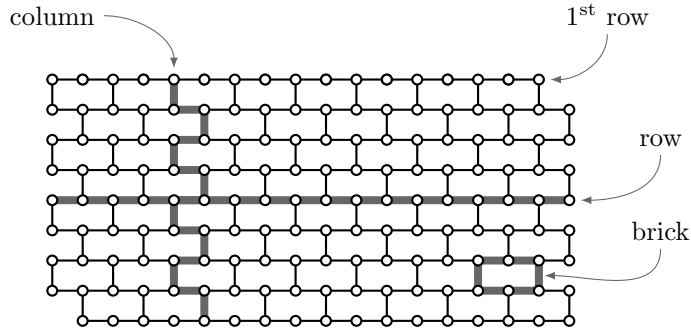

A \emph{wall} is defined as the subdivision of an elementary wall. The definition of rows, columns, boundary and bricks in an elementary wall carries over to a wall in a natural way (with some truncation of the first and last row and column). For brevity of notation, we define an \emph{$n$-wall} as a wall of size at least~$n\times n$. 


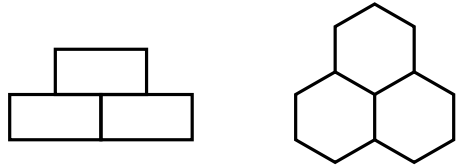
\begin{figure}[hbt] 
\centering
\begin{tikzpicture}[scale=0.6]
\tikzstyle{tinyvx}=[thick,circle,inner sep=0.cm, minimum size=1.5mm, fill=white, draw=black]

\def\side{1}
\def\height{0.5} 
\def\width{1.732} 
\def\halfwidth{0.5*\width}
\def\heightplusside{\height+\side}

\def\offset{3*\side}

\begin{scope}[shift={(-\offset,0)}]

	\draw[hedge] (0,0) rectangle (-2*\side,\side) ;
	\draw[hedge] (0,0) rectangle ( 2*\side,\side) ;
	\draw[hedge] (-\side,\side) rectangle (\side,2*\side) ;

\end{scope}

\begin{scope}[shift={(\offset,0)}]

	\draw[hedge] (0,0) -- (0,\side) -- (\halfwidth,\heightplusside) -- (\width,\side) -- (\width,0) -- (\halfwidth,-\height) -- (0,0);
	
	\draw[hedge] (0,\side) -- (-\halfwidth,\heightplusside) -- (-\width,\side) -- (-\width,0) -- (-\halfwidth,-\height) -- (0,0);
	
	\draw[hedge] (-\halfwidth,\heightplusside) -- (-\halfwidth,\heightplusside+\side) -- (0,2*\height+2*\side) -- (\halfwidth,\heightplusside+\side) -- (\halfwidth,\heightplusside);
	
\end{scope}

\end{tikzpicture}
\caption{Three bricks that might be part of a wall. Left, you see a plain drawing with bricks in the form of rectangles. Right, you see the same part of a wall but with bricks in the form of honeycombs. In this article, we will draw walls as on the right side to underline symmetry.}
\label{fig:3BrickPlainAndHoneycomb}
\end{figure}

When analyzing small parts of walls, we will not only count the number of bricks they contain, but also specify how the bricks are connected. To this end, we define $B_1, B_2, \ldots , B_{10}$ to be subdivisions of the graphs shown in Figure~\ref{fig:B_n_definition}.

\begin{figure}[hbt] 
\centering
\begin{tikzpicture}[scale=0.4]
\tikzstyle{tinyvx}=[thick,circle,inner sep=0.cm, minimum size=1.5mm, fill=white, draw=black]

\def\side{1}
\def\height{0.5} 
\def\width{1.732} 
\def\halfwidth{0.5*\width}
\def\heightplusside{\height+\side}

\def\hoffset{1.5*\width} 
\def\voffset{2*\side} 

\begin{scope}[shift={(-\width,0)}] 

	\draw[hedge] (0,0) -- (0,\side) -- (\halfwidth,\heightplusside) -- (\width,\side) -- (\width,0) -- (\halfwidth,-\height) -- (0,0);
	\draw (\halfwidth,-\height) node[below] {$B_1$};
\end{scope}

\begin{scope}[shift={(\hoffset + \width,0)}] 
	\draw[hedge] (0,0) -- (0,\side) -- (\halfwidth,\heightplusside) -- (\width,\side) -- (\width,0) -- (\halfwidth,-\height) -- (0,0);
	\draw[hedge] (0,\side) -- (-\halfwidth,\heightplusside) -- (-\width,\side) -- (-\width,0) -- (-\halfwidth,-\height) -- (0,0);
	\draw (0,-\height) node[below] {$B_2$};
\end{scope}

\begin{scope}[shift={(2*\hoffset+3*\width,0)}] 
	\draw[hedge] (0,0) -- (0,\side) -- (\halfwidth,\heightplusside) -- (\width,\side) -- (\width,0) -- (\halfwidth,-\height) -- (0,0);
	\draw[hedge] (0,\side) -- (-\halfwidth,\heightplusside) -- (-\width,\side) -- (-\width,0) -- (-\halfwidth,-\height) -- (0,0);
	\draw[hedge] (-\halfwidth,\heightplusside) -- (-\halfwidth,\heightplusside+\side) -- (0,2*\height+2*\side) -- (\halfwidth,\heightplusside+\side) -- (\halfwidth,\heightplusside);
	\draw (0,-\height) node[below] {$B_3$};
\end{scope}

\begin{scope}[shift={(3*\hoffset+5*\width,\heightplusside)}] 
	\draw[hedge] (0,0) -- (0,\side) -- (\halfwidth,\heightplusside) -- (\width,\side) -- (\width,0) -- (\halfwidth,-\height) -- (0,0);
	\draw[hedge] (0,\side) -- (-\halfwidth,\heightplusside) -- (-\width,\side) -- (-\width,0) -- (-\halfwidth,-\height) -- (0,0);
	\draw[hedge] (-\halfwidth,\heightplusside) -- (-\halfwidth,\heightplusside+\side) -- (0,2*\height+2*\side) -- (\halfwidth,\heightplusside+\side) -- (\halfwidth,\heightplusside);
	\draw[hedge] (-\halfwidth,-\height) -- (-\halfwidth,-\height-\side) -- (0,-2*\height-\side) -- (\halfwidth,-\height-\side) -- (\halfwidth,-\height);
	\draw (0,-2*\height-\side) node[below] {$B_4$};
\end{scope}


\begin{scope}[shift={(0,-\height-\voffset-2*\height-2*\side)}] 
	\draw[hedge] (0,0) -- (0,\side) -- (\halfwidth,\heightplusside) -- (\width,\side) -- (\width,0) -- (\halfwidth,-\height) -- (0,0);
	\draw[hedge] (0,\side) -- (-\halfwidth,\heightplusside) -- (-\width,\side) -- (-\width,0) -- (-\halfwidth,-\height) -- (0,0);
	\draw[hedge] (-\halfwidth,\heightplusside) -- (-\halfwidth,\heightplusside+\side) -- (0,2*\height+2*\side) -- (\halfwidth,\heightplusside+\side) -- (\halfwidth,\heightplusside);
	\draw[hedge] (-\halfwidth,-\height) -- (-\halfwidth,-\height-\side) -- (0,-2*\height-\side) -- (\halfwidth,-\height-\side) -- (\halfwidth,-\height);
	\draw[hedge] (\halfwidth,-\height-\side) -- (\width,-2*\height-\side) -- (\width+\halfwidth,-\height-\side) -- (\width+\halfwidth,-\height) -- (\width,0);
	\draw (\halfwidth,-2*\height-\side) node[below] {$B_5$};
\end{scope}

\begin{scope}[shift={(\hoffset+3*\width,-\height-\voffset-2*\height-2*\side)}] 
	\draw[hedge] (0,0) -- (0,\side) -- (\halfwidth,\heightplusside) -- (\width,\side) -- (\width,0) -- (\halfwidth,-\height) -- (0,0);
	\draw[hedge] (0,\side) -- (-\halfwidth,\heightplusside) -- (-\width,\side) -- (-\width,0) -- (-\halfwidth,-\height) -- (0,0);
	\draw[hedge] (-\halfwidth,\heightplusside) -- (-\halfwidth,\heightplusside+\side) -- (0,2*\height+2*\side) -- (\halfwidth,\heightplusside+\side) -- (\halfwidth,\heightplusside);
	\draw[hedge] (-\halfwidth,-\height) -- (-\halfwidth,-\height-\side) -- (0,-2*\height-\side) -- (\halfwidth,-\height-\side) -- (\halfwidth,-\height);
	\draw[hedge] (\halfwidth,-\height-\side) -- (\width,-2*\height-\side) -- (\width+\halfwidth,-\height-\side) -- (\width+\halfwidth,-\height) -- (\width,0);
	\draw[hedge] (\width,\side) -- (\width+\halfwidth,\heightplusside) -- (\width+\halfwidth,\heightplusside+\side) -- (\width,2*\side+2*\height) -- (\halfwidth,\heightplusside+\side);
	\draw (\halfwidth,-2*\height-\side) node[below] {$B_6$};
\end{scope}

\begin{scope}[shift={(2*\hoffset+6*\width,-\height-\voffset-2*\height-2*\side)}] 
	\draw[hedge] (0,0) -- (0,\side) -- (\halfwidth,\heightplusside) -- (\width,\side) -- (\width,0) -- (\halfwidth,-\height) -- (0,0);
	\draw[hedge] (0,\side) -- (-\halfwidth,\heightplusside) -- (-\width,\side) -- (-\width,0) -- (-\halfwidth,-\height) -- (0,0);
	\draw[hedge] (-\halfwidth,\heightplusside) -- (-\halfwidth,\heightplusside+\side) -- (0,2*\height+2*\side) -- (\halfwidth,\heightplusside+\side) -- (\halfwidth,\heightplusside);
	\draw[hedge] (-\halfwidth,-\height) -- (-\halfwidth,-\height-\side) -- (0,-2*\height-\side) -- (\halfwidth,-\height-\side) -- (\halfwidth,-\height);
	\draw[hedge] (\halfwidth,-\height-\side) -- (\width,-2*\height-\side) -- (\width+\halfwidth,-\height-\side) -- (\width+\halfwidth,-\height) -- (\width,0);
	\draw[hedge] (\width,\side) -- (\width+\halfwidth,\heightplusside) -- (\width+\halfwidth,\heightplusside+\side) -- (\width,2*\side+2*\height) -- (\halfwidth,\heightplusside+\side);
	\draw[hedge] (\width+\halfwidth,-\height) -- (\width+2*\halfwidth,0) -- (\width+2*\halfwidth,\side) -- (\width+\halfwidth,\heightplusside);
	\draw (\halfwidth,-2*\height-\side) node[below] {$B_7$};
\end{scope}


\begin{scope}[shift={(0,-\height-3*\side-4*\height-2*\voffset-2*\height-2*\side)}] 
	\draw[hedge] (0,0) -- (0,\side) -- (\halfwidth,\heightplusside) -- (\width,\side) -- (\width,0) -- (\halfwidth,-\height) -- (0,0);
	\draw[hedge] (0,\side) -- (-\halfwidth,\heightplusside) -- (-\width,\side) -- (-\width,0) -- (-\halfwidth,-\height) -- (0,0);
	\draw[hedge] (-\halfwidth,\heightplusside) -- (-\halfwidth,\heightplusside+\side) -- (0,2*\height+2*\side) -- (\halfwidth,\heightplusside+\side) -- (\halfwidth,\heightplusside);
	\draw[hedge] (-\halfwidth,-\height) -- (-\halfwidth,-\height-\side) -- (0,-2*\height-\side) -- (\halfwidth,-\height-\side) -- (\halfwidth,-\height);
	\draw[hedge] (\halfwidth,-\height-\side) -- (\width,-2*\height-\side) -- (\width+\halfwidth,-\height-\side) -- (\width+\halfwidth,-\height) -- (\width,0);
	\draw[hedge] (\width,\side) -- (\width+\halfwidth,\heightplusside) -- (\width+\halfwidth,\heightplusside+\side) -- (\width,2*\side+2*\height) -- (\halfwidth,\heightplusside+\side);
	\draw[hedge] (\width+\halfwidth,-\height) -- (\width+2*\halfwidth,0) -- (\width+2*\halfwidth,\side) -- (\width+\halfwidth,\heightplusside);
	\draw[hedge] (\width+\halfwidth,-\height-\side) -- (2*\width,-2*\height-\side) -- (2*\width+\halfwidth,-\height-\side) -- (2*\width+\halfwidth,-\height) -- (2*\width,0);
	\draw (\width,-2*\height-\side) node[below] {$B_8$};
\end{scope}

\begin{scope}[shift={(\hoffset+3*\width,-\height-3*\side-4*\height-2*\voffset-2*\height-2*\side)}] 
	\draw[hedge] (0,0) -- (0,\side) -- (\halfwidth,\heightplusside) -- (\width,\side) -- (\width,0) -- (\halfwidth,-\height) -- (0,0);
	\draw[hedge] (0,\side) -- (-\halfwidth,\heightplusside) -- (-\width,\side) -- (-\width,0) -- (-\halfwidth,-\height) -- (0,0);
	\draw[hedge] (-\halfwidth,\heightplusside) -- (-\halfwidth,\heightplusside+\side) -- (0,2*\height+2*\side) -- (\halfwidth,\heightplusside+\side) -- (\halfwidth,\heightplusside);
	\draw[hedge] (-\halfwidth,-\height) -- (-\halfwidth,-\height-\side) -- (0,-2*\height-\side) -- (\halfwidth,-\height-\side) -- (\halfwidth,-\height);
	\draw[hedge] (\halfwidth,-\height-\side) -- (\width,-2*\height-\side) -- (\width+\halfwidth,-\height-\side) -- (\width+\halfwidth,-\height) -- (\width,0);
	\draw[hedge] (\width,\side) -- (\width+\halfwidth,\heightplusside) -- (\width+\halfwidth,\heightplusside+\side) -- (\width,2*\side+2*\height) -- (\halfwidth,\heightplusside+\side);
	\draw[hedge] (\width+\halfwidth,-\height) -- (\width+2*\halfwidth,0) -- (\width+2*\halfwidth,\side) -- (\width+\halfwidth,\heightplusside);
	\draw[hedge] (\width+\halfwidth,-\height-\side) -- (2*\width,-2*\height-\side) -- (2*\width+\halfwidth,-\height-\side) -- (2*\width+\halfwidth,-\height) -- (2*\width,0);
	\draw[hedge] (2*\width,\side) -- (2*\width+\halfwidth,\heightplusside) -- (2*\width+\halfwidth,\heightplusside+\side) -- (2*\width,2*\side+2*\height) -- (\width+\halfwidth,\heightplusside+\side);
	\draw (\width,-2*\height-\side) node[below] {$B_9$};
\end{scope}

\begin{scope}[shift={(2*\hoffset+6*\width,-\height-3*\side-4*\height-2*\voffset-2*\height-2*\side)}] 
	\draw[hedge] (0,0) -- (0,\side) -- (\halfwidth,\heightplusside) -- (\width,\side) -- (\width,0) -- (\halfwidth,-\height) -- (0,0);
	\draw[hedge] (0,\side) -- (-\halfwidth,\heightplusside) -- (-\width,\side) -- (-\width,0) -- (-\halfwidth,-\height) -- (0,0);
	\draw[hedge] (-\halfwidth,\heightplusside) -- (-\halfwidth,\heightplusside+\side) -- (0,2*\height+2*\side) -- (\halfwidth,\heightplusside+\side) -- (\halfwidth,\heightplusside);
	\draw[hedge] (-\halfwidth,-\height) -- (-\halfwidth,-\height-\side) -- (0,-2*\height-\side) -- (\halfwidth,-\height-\side) -- (\halfwidth,-\height);
	\draw[hedge] (\halfwidth,-\height-\side) -- (\width,-2*\height-\side) -- (\width+\halfwidth,-\height-\side) -- (\width+\halfwidth,-\height) -- (\width,0);
	\draw[hedge] (\width,\side) -- (\width+\halfwidth,\heightplusside) -- (\width+\halfwidth,\heightplusside+\side) -- (\width,2*\side+2*\height) -- (\halfwidth,\heightplusside+\side);
	\draw[hedge] (\width+\halfwidth,-\height) -- (\width+2*\halfwidth,0) -- (\width+2*\halfwidth,\side) -- (\width+\halfwidth,\heightplusside);
	\draw[hedge] (\width+\halfwidth,-\height-\side) -- (2*\width,-2*\height-\side) -- (2*\width+\halfwidth,-\height-\side) -- (2*\width+\halfwidth,-\height) -- (2*\width,0);
	\draw[hedge] (2*\width,\side) -- (2*\width+\halfwidth,\heightplusside) -- (2*\width+\halfwidth,\heightplusside+\side) -- (2*\width,2*\side+2*\height) -- (\width+\halfwidth,\heightplusside+\side);
	\draw[hedge] (2*\width+\halfwidth,-\height) -- (2*\width+2*\halfwidth,0) -- (2*\width+2*\halfwidth,\side) -- (2*\width+\halfwidth,\heightplusside);
	\draw (\width,-2*\height-\side) node[below] {$B_{10}$};
\end{scope}

\end{tikzpicture}
\caption{Up to isomorphism, these graphs are defined to be the only elementary $B_n$ for $n \leq 10$.}
\label{fig:B_n_definition}
\end{figure}
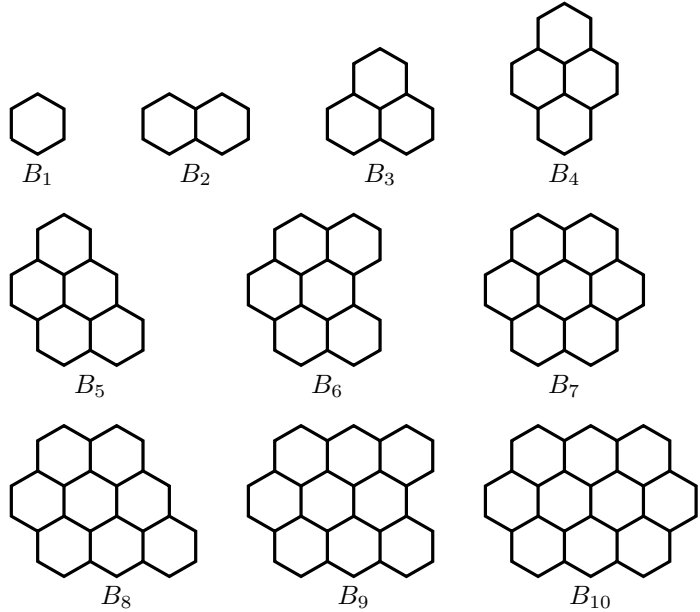

\section{Construction}

For every graph $G$ and every $r \in \N$, an \textbf{$r$-fold $G$} is a graph $G'$ where every edge in $G$ is replaced by $r$ edge-disjoint paths of length~$2$.

Given a size $f(2)$ of some hypothetical hitting set, we choose $r \in \N$ with $r > f(2)$. Let $B$ be a wall of size at least $6 \times 4$. Note that for every wall, there are exactly two bricks which are adjacent to only two other bricks, while all others are adjacent to at least three bricks.
We define the \emph{body} of $B$ to be the minimal subgraph of $B$ that contains all bricks with at least three adjacent bricks in $B$. In other words, the body of $B$ contains everything but the two less connected bricks in the corners of $B$.

In the body of $B$, we pick two edges $e_1, e_2$ on the outer face of $B$ such that every $e_1$-$e_2$~path $P$ in $B$ is incident with at least $7$~bricks of $B$ aside from the ones containing $e_1$ or $e_2$. This is possible since $B$ has size at least $6 \times 4$, so there are at least $16$ bricks adjacent to the outer face of $B$.

Now we define $G^*$ to be the graph consisting of the union of an $r$-fold $B - \{e_1,e_2\}$ and a condensed wall $W$ of size~$r$ whose terminals $a, b, c$ and $d$ are identified with the endvertices of $e_1$ and $e_2$ such that $a$ and $b$ are incident with the same edge and every $ac$-path in $G^* - (W - \{a,b,c,d\})$ disconnects $b$ from $d$ in $G^* - (W - \{a,b,c,d\})$.

\section{Results}

\subsection{Strategy}

First we check that the construction of $G^*$ does not allow for an edge hitting set of size less than $r$.

\begin{lemma} \label{wall:lem:noEdgeHittingSet}
For every set $F \subseteq E(G^*)$ with $|F| \leq r-1$, $G^* - F$ contains a $B$-expansion.
\end{lemma}

\begin{proof}
By construction, $G^* - W^0$ contains an $r$-fold $B - \{e_1,e_2\}$, so $G^* - W^0 - F$ still contains an embedding of $B - \{e_1, e_2\}$. Furthermore, $W - F$ contains an ($a$--$b$, $c$--$d$) linkage. Together, this yields a $B$-expansion in $G^* - F$.
\end{proof}

This allows us to prove the first half of Theorem~\ref{wall:thm:WallsHaveNotEdgeErdosPosaIf6x4Contained}.

\begin{proof}
Let $B$ be a wall of size $6 \times 4$, and let $G^*$ be as described above. By Lemma~\ref{wall:lem:noEdgeHittingSet}, there is no edge hitting set. We thus only have to show that there can be no two edge-disjoint embeddings of $B$ in $G^*$

Now let $U$ be an arbitrary embedding of $B$ in $G^*$. Suppose we were able to show that $U$ must contain an ($a$--$b$, $c$--$d$) linkage in $W$. By Lemma~\ref{lem:noTwoLinkages} there can be no two edge-disjoint linkages in $W$, implying there can be no two edge-disjoint embeddings $U_1, U_2$ of $B$ in $G^*$. This implies that $B$ does not have the edge-Erd\H{o}s-P\'{o}sa property, which finishes our proof.
\end{proof}

For the remainder of this chapter, let $U$ be a fixed embedding of $B$ in $G^*$. It remains to show that:

\begin{lemma} \label{wall:lem:linkageMainLemma}
$U$ contains an ($a$--$b$, $c$--$d$) linkage in $W$.
\end{lemma}

\subsection{Embedding Walls in a Heinlein Wall}

To be able to control the embedding of $U$, we need to make sure that no large part of it can be embedded in $W$. To this end, we study which small walls are too large to fit in $W$.
For our first lemmas, we will label some parts of a $B_3$ (see Figure~\ref{fig:3BrickWallLabels}). Let $x$ be the unique vertex in $B_3$ whose neighbors (in an elementary wall) have degree~$3$.
Let the other vertices with degree three be called $w, y$ and $z$. Let the brick that is not incident with $w$ be called $C_1$. Let $P_1, P_2$ and $P_3$ be the $yz$-, $zw$- and $wy$-paths not crossing $x$.

\begin{figure}[hbt] 
\centering
\begin{tikzpicture}[scale=0.6]
\tikzstyle{tinyvx}=[thick,circle,inner sep=0.cm, minimum size=1.5mm, fill=white, draw=black]

\def\side{1}
\def\height{0.5} 
\def\width{1.732} 
\def\halfwidth{0.5*\width}
\def\heightplusside{\height+\side}

\draw[hedge] (0,0) -- (0,\side) -- (\halfwidth,\heightplusside) -- (\width,\side) -- (\width,0) -- (\halfwidth,-\height) -- (0,0);

\draw[hedge] (0,\side) -- (-\halfwidth,\heightplusside) -- (-\width,\side) -- (-\width,0) -- (-\halfwidth,-\height) -- (0,0);

\draw[hedge] (-\halfwidth,-\height) -- (-\halfwidth,-\height-\side) -- (0,-\height-\side-\height) -- (\halfwidth,-\height-\side) -- (\halfwidth,-\height);

\node[smallvx, label=right:$x$] (x) at (0,0) {};
\node[smallvx, label=above:$w$] (w) at (0,\side) {};
\node[smallvx, label=right:$y$] (y) at (\halfwidth,-\height) {};
\node[smallvx, label=left:$z$] (z) at (-\halfwidth,-\height) {};
\node at (0,-\height-0.5*\side) {$C_1$};
\node[label=below:$P_1$] at (0,-2*\height-\side) {};
\node[label=left:$P_2$] at (-\width,\side) {};
\node[label=right:$P_3$] at (\width,\side) {};

\end{tikzpicture}
\caption{A $B_3$ with labels.}
\label{fig:3BrickWallLabels}
\end{figure}
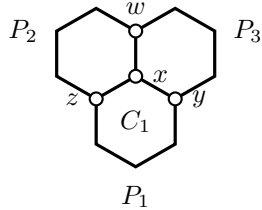

\begin{lemma} \label{wall:lem:centralVertexIsBottleneck}
For every $B_3$ in $W - \{a, b\}$, $x$ is a bottleneck vertex.
\end{lemma}

\begin{proof}
Suppose $x$ would not be a bottleneck vertex. Then $x$ must be on the row $R_i$ in some layer $W_i$, and thus $x$ has only degree~$3$ in $W$.
We observe that $x$ is adjacent with a bottleneck vertex, say $z_i$. Therefore, one of the edges incident with $x$ is also incident with $z_i$. Therefore, we can assume that one of $w, y$ and $z$ (say $w$) is identical to $z_{i}$ or $z_{i}$ lies on the $wx$-path not crossing $y$ or $z$. In any case, $z_i$ is not part of the brick $C_1$.

As $C_1$ contains $x$ and is a cycle, it must also contain $z_{i-1}$. Additionally, $y$ and $z$ lie on $C_1$ but cannot both be incident with $z_{i}$ as otherwise there would be only one instead of three (or more) vertices between them. Say $y$ is not incident with $z_{i}$. Then both neighbors of $y$ on $R_i$ lie on $C_1$. But $y$ also has a connection $P_3$ of length at least three to $w$ that is internally disjoint from $C_1$. However, $P_3$ may only contain the edge $yz_i$, which is too short.
\end{proof}

We could directly apply Lemma~\ref{wall:lem:centralVertexIsBottleneck} to see that a $5$-brick wall $B_5$ cannot be embedded in $W - \{a, b\}$, because it would contain three different vertices that are a central vertex $x$ in a $B_3$, but there are only two bottleneck vertices in a layer of $W$. We will later see that we can even exclude a $4$-brick wall $B_4$, but for this we will need to see how a $3$-brick wall $B_3$ can be embedded in $W - \{a, b\}$.

\begin{lemma}
For all integers $r$ and $n$, a condensed wall without jump-edges $W^-$ of size $r + n$ contains $n$ edge-disjoint embeddings of $B_3$ in $W^- - \{a, b\}$, even with $r$ edges being deleted.
\end{lemma}

\begin{proof}
Obviously, there still exist $n$ complete layers of $W^-$. In each, we will place one $B_3$ as in Figure~\ref{fig:3BrickWallInLayer}, so they are all edge-disjoint because the layers of $W^-$ are.
\end{proof}

\begin{figure}[hbt] 
\centering
\begin{tikzpicture}[scale=1.2]
\tikzstyle{tinyvx}=[thick,circle,inner sep=0.cm, minimum size=1.5mm, fill=white, draw=black]

\def\vstep{2}
\def\hstep{0.5}
\def\hwidth{11}
\def\hheight{0}

\def\totalheight{\hheight*\vstep}
\def\totalwidth{\hwidth*\hstep}
\pgfmathtruncatemacro{\minustwo}{\hwidth-2}
\pgfmathtruncatemacro{\minusone}{\hwidth-1}

\foreach \j in {0,...,\hheight} {
	\draw[hedge] (0,\j*\vstep) -- (\hwidth*\hstep,\j*\vstep);
	\foreach \i in {0,...,\hwidth} {
		\node[tinyvx] (v\i\j) at (\i*\hstep,\j*\vstep){};
	}
}
\pgfmathtruncatemacro{\plusvone}{\hheight+1}

\foreach \j in {0,...,\plusvone}{
	\node[hvertex] (z\j) at (0.5*\hwidth*\hstep,-\j*\vstep+0.5*\vstep) {};
}



\foreach \j in {1,...,\plusvone}{
	\pgfmathtruncatemacro{\subone}{\j-1}
	\foreach \i in {1,3,...,\hwidth}{
		\draw[hedge] (z\j) to (v\i\subone);
	}
}

\foreach \j in {0,...,\hheight}{
	\foreach \i in {0,2,...,\hwidth}{
		\draw[hedge] (z\j) to (v\i\j);
	}
}


%

\begin{pgfonlayer}{foreground}
	\draw[red, hedge] (z0)  to (v20) ;
	\draw[red, hedge] (z0)  to (v60) ;
	\draw[red, hedge] (z0)  to (v100) ;
	\foreach \i in {1,...,\minusone}{
		\pgfmathtruncatemacro{\plusone}{\i+1}
		\draw[red, hedge] (v\i0) to (v\plusone0);
	}
	\draw[red, hedge] (z1)  to (v10) ;
	\draw[red, hedge] (z1)  to (v110) ;

\end{pgfonlayer}

\end{tikzpicture}
\caption{A $B_3$ in a layer of a modified condensed wall $W^-$ without jump-edges.}
\label{fig:3BrickWallInLayer}
\end{figure}
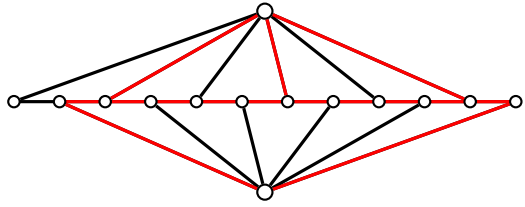

\begin{lemma} \label{wall:lem:noB4InSingleLayer}
There exists no $B_4$ in $W - \{a, b\}$.
\end{lemma}

\begin{proof}
First, we focus on a fixed $B_3$ which is a subgraph of every $B_4$. By Lemma~\ref{wall:lem:centralVertexIsBottleneck}, we know that $x$ must be a bottleneck vertex, say $z_{i-1}$. As a wall is $2$-connected, we conclude that the rest of $B_3$ must be in the same layer $W_i$ of $W$ as $x$. Where are $w, y$ and $z$?

$w, y$ and $z$ could all lie on the row $R_i$ of $W_i$. The outer two (say $w$ and $y$) must then be connected via the path $P_3$ that uses $z_i$. This implies that $P_1$ and $P_2$ must be entirely contained in $R_i$. Now there is no connection between any two of $P_1, P_2$ and $P_3$ in $W_i - B_3$, which implies there can no $B_4$ in $W - \{a, b\}$ which contains a $B_3$ that is embedded as above.

The only alternative for a $B_3$ would be to use $z_i$ for one of $w, y$ and $z$, say $w$. But then $z$ and $y$ are still on $R_i$. As both bottleneck vertices are blocked, $P_1$ must be entirely contained in $R_i$, too. Now $y$ and $z$ separate $P_1, P_2$ and $P_3$ in $R_i$, so there cannot be any path connecting any two of them in $W - \{a, b\}$ without using $z_{i-1}$ or $z_i$, which are blocked by $x$ and $w$. Again, we conclude that there cannot be a $B_4$.
\end{proof}

\begin{lemma} \label{wall:lem:noUseOfB3InSingleLayer}
There exists no $B_4$ in $W^-$ that contains a $B_3$ in $W - \{a, b\}$. The same holds true for a $B_5$ in $W$.
\end{lemma}

\begin{proof}
In the proof of Lemma~\ref{wall:lem:noB4InSingleLayer}, we have seen that there are only two possibilities to embed a $B_3$ in $W - \{a, b\}$. Let us start with the first one, i. e. $w, y$ and $z$ are all on the same row $R_i$, and $P_3$ connects $w$ and $y$ via $z_i$. Then there is no path between $P_1$ or $P_2$ and a vertex outside of $R_i$ that does not cross $B_3$, which implies $W$ cannot contain a $B_4$.

So let us assume $w$ is identical to $z_i$ instead. In $W^- - \{a, b\}$, the subdivision of edge~$1$ crosses $R_i$, so both $y$ and $z$ must be on the same side of it. But then $B_3$ can have only one path connecting it to a vertex outside of $W_i$ via either $u^i_1$ or $u^i_{2r}$, but two such connections would be needed to form an additional cycle and thus a $B_4$. In $W$ (instead of $W^-$), there can be both connections, but this only enables a $B_4$ and not a $B_5$.
\end{proof}

The last lemma shows that no large wall can contain a $B_3$ in $W - \{a, b\}$. Next, we will see what happens for a $B_2$.

\begin{lemma} \label{wall:lem:B2UsesBottleneckVertex}
Every $B_2$ in $W - \{a, b\}$ that is part of a $B_3$ in $W$ contains a bottleneck vertex as a vertex of degree~$3$.
\end{lemma}

\begin{proof}
As $B_2$ is $2$-connected, $B_2$ must be contained in a single layer $W_i$ of $W$.
We start by labeling our $B_2$ as in Figure~\ref{fig:2BrickWallLabels}. Let $C_1$ and $C_2$ be the two bricks of $B_2$, and let $x$ and $y$ be the vertices of degree~$3$. Let $Q$ be the $xy$-path belonging to both $C_1$ and $C_2$, and let $P_1$ be the $xy$-path which is unique to $C_1$, while $P_2$ is the $xy$-path that lies only in $C_2$. Finally, let $P_3$ be the path in $B_3 - B_2$ connecting $P_1$ and $P_2$. Suppose $x$ and $y$ would both not be a bottleneck vertex. Then $x$ and $y$ lie both on the row $R_i$. Additionally, by Lemma~\ref{wall:lem:centralVertexIsBottleneck}, the $B_3$ cannot be entirely contained in $W - \{a,b\}$.

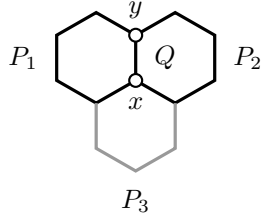
\begin{figure}[hbt] 
\centering
\begin{tikzpicture}[scale=0.6]
\tikzstyle{tinyvx}=[thick,circle,inner sep=0.cm, minimum size=1.5mm, fill=white, draw=black]

\def\side{1}
\def\height{0.5} 
\def\width{1.732} 
\def\halfwidth{0.5*\width}
\def\heightplusside{\height+\side}

\draw[hedge] (0,0) -- (0,\side) -- (\halfwidth,\heightplusside) -- (\width,\side) -- (\width,0) -- (\halfwidth,-\height) -- (0,0);

\draw[hedge] (0,\side) -- (-\halfwidth,\heightplusside) -- (-\width,\side) -- (-\width,0) -- (-\halfwidth,-\height) -- (0,0);

\draw[grau, hedge] (-\halfwidth,-\height) -- (-\halfwidth,-\height-\side) -- (0,-\height-\side-\height) -- (\halfwidth,-\height-\side) -- (\halfwidth,-\height);

\node[smallvx, label=below:$x$] (x) at (0,0) {};
\node[smallvx, label=above:$y$] (y) at (0,\side) {};
\node[label=right:$Q$] at (0,0.5*\side) {};
\node[label=below:$P_3$] at (0,-2*\height-\side) {};
\node[label=left:$P_1$] at (-\width,0.5*\side) {};
\node[label=right:$P_2$] at (\width,0.5*\side) {};

\end{tikzpicture}
\caption{A $B_2$ with labels. The grey path $P_3$ belongs to the $B_3$.}
\label{fig:2BrickWallLabels}
\end{figure}

\vspace{1ex}
\noindent\textbf{Case~1: $Q$ contains no bottleneck vertex}

Suppose $Q$ would not contain a bottleneck vertex. Then $Q$ is entirely contained in $R_i$. As $x$ and $y$ have degree~$3$ in $B_2$, as is the maximum degree in $R_i$, every edge incident with $x$ or $y$ in $W$ must also be used for $B_2$. In particular, the paths $P_1$ and $P_2$ use the edges connecting $x$ and $y$ with $z_{i-1}$ and $z_i$. From those bottleneck vertices, $P_1$ and $P_2$ must use an edge to $R_i$, and from there on they can only contain a path connecting them to $x$ or $y$ on $R_i$.
Now the only possibilities to connect $P_1$ or $P_2$ to a vertex outside of $W_i$ would be to use the bottleneck vertices, which are the first vertex on $P_1$ after $x$ and $y$, or the very next vertex which they hit on the row, which comes second on $P_1$ and $P_2$. However, none of $P_1$ and $P_2$ can have a connection via its third (or a later) vertex. Note that we counted the vertices for both $P_1$ and $P_2$ in the same order (e. g. clockwise).

When looking at how a $B_2$ can be extended to a $B_3$, we notice that the first branch vertex of $P_1$ must be connected with the third branch vertex of $P_2$, or vice versa. This is impossible as we observed above, a contradiction.

\vspace{1ex}
\noindent\textbf{Case~2: $Q$ contains a bottleneck vertex}

Suppose $Q$ would contain at least one bottleneck vertex, say $z_i$. As we required $x$ and $y$ not to be a bottleneck vertex, we conclude that $z_i$ is an interior vertex of $Q$. Now $P_1$ and $P_2$ must connect $x$ and $y$ in $W_i$ without using $z_i$. Therefore, one of them (say $P_1$) must be entirely contained in $R_i$. Furthermore, $P_2$ and $Q$ separate  $P_1$ from $W - W_i$. This implies $B_2$ cannot be extended to a $B_3$, a contradiction.
\end{proof}

\begin{lemma} \label{wall:lem:atMost3ConnectionsForB2}
Let $G$ be a subcubic subgraph of a condensed wall $W$ that contains a $B_2$ in $W - \{a, b\}$ (which implies $B_2$ is contained in a single layer $W_i$ of $W$). Then there are at~most~$3$ disjoint paths in $G$ connecting $B_2$ to $W - W_i$, and at most one of them may use a bottleneck vertex of $W_i$.
\end{lemma}

\begin{proof}
By Lemma~\ref{wall:lem:B2UsesBottleneckVertex}, we know that $B_2$ must already contain one of the bottleneck vertex as a vertex of degree~$3$. As $G$ is subcubic, this vertex cannot be used for a connection of $B_2$ to $W - W_i$. This leaves only one bottleneck vertex and $a$ and $b$ for such connections.
\end{proof}

\begin{defin} 
A $B_1^2$ is the union of two disjoint $B_1$ ($H_1$, $H_2$), and two disjoint paths ($P_1, P_2$) each connecting $H_1$ with $H_2$, i.e. $B_1^2 := H_1 \cup H_2 \cup P_1 \cup P_2$.
\end{defin}

\begin{lemma} \label{wall:lem:noB7inCondensedWall}
In a $B_7$, let $Q$ be a disjoint (and possibly empty) union of paths $P$ in $B_7$, such that all interior vertices of $P$ have degree at~most~$2$ in $B_7$, and let $P$ have only bottleneck vertices and at least one of $a$ and $b$ as endvertices. Then there is no embedding of $B_7 - Q$ in a condensed wall $W$.
\end{lemma}

\begin{proof}
Suppose $B_7 - Q$ would be in $W$. First, we notice that every $B_2$ in $B_7 - Q$ has at least four disjoint connections to $B_7 - B_2$. Therefore, we conclude with Lemma~\ref{wall:lem:atMost3ConnectionsForB2} that no $B_2$ can be contained in $W - \{a, b\}$.

Next, we observe that each of $a$ and $b$ can be only situated on at most two of the six outer bricks of $B_7$. This implies $B_7 - \{a, b\} - Q$ contains at least two disjoint $B_1$, call them $H_1$ and $H_2$. Additionally, we can conclude that $H_1$ and $H_2$ were on opposite sides of $B_7$, otherwise we would get a $B_2$ in $W - \{a, b\}$. Furthermore, $H_1$ and $H_2$ were connected by four disjoint paths in $B_7$, meaning there are still at least two disjoint paths $P_1$, $P_2$ connecting $H_1$ and $H_2$ in $B_7 - \{a, b\} - Q$. (Note that all paths in $Q$ begin or end in either $a$ or $b$.) 

Now $U = H_1 \cup H_2 \cup P_1 \cup P_2$ is a $B_1^2$ in $W - \{a, b\}$. First, we notice that $U$ is $2$-connected, which implies that $U$ is contained in a single layer $W_i$ of $W$. Next, we observe that there are $6$~connections of $U$ to $a$ and $b$ in $B_7$. Therefore, $U$ must have $6$ connecting vertices (i.e. bottleneck vertices or vertices adjacent to $a$ or $b$). However, there can be at most $4$ connecting vertices, a contradiction.
\end{proof}

\begin{lemma} \label{wall:lem:manyB6inCondensedWall}
For every $n, r \in\N$, in every condensed wall without jump-edges $W^-$ of size $3 \cdot (n + r)$, there are $n$~edge-disjoint $B_6$ in $W^-$, even with $r$ edges of $W^-$ being deleted.
\end{lemma}

\begin{proof}
Even with $r$ edges of $W^-$ being deleted, there are still at least $n$ untouched edge-disjoint chunks of $W^-$ that each contain three consecutive layers of $W^-$ and which all have an edge to both $a$ and $b$. In each of them, we can find a $B_6$ as in Figure~\ref{fig:6BrickWallInCondWall}.
\end{proof}

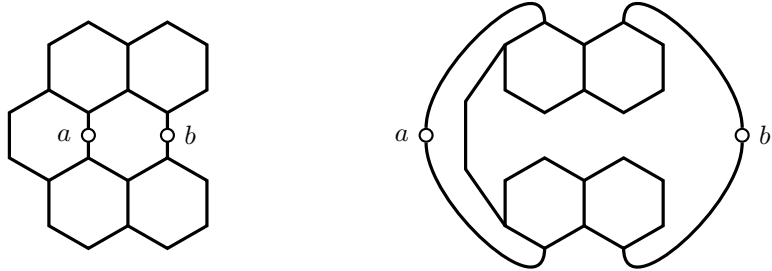
\begin{figure}[hbt] 
\centering
\begin{tikzpicture}[scale=0.6]
\tikzstyle{tinyvx}=[thick,circle,inner sep=0.cm, minimum size=1.5mm, fill=white, draw=black]

\def\side{1}
\def\height{0.5} 
\def\width{1.732} 
\def\halfwidth{0.5*\width}
\def\heightplusside{\height+\side}

\def\offset{5*\side}

\begin{scope}[shift={(-\offset,0)}]

	\draw[hedge] (0,0) -- (0,\side) -- (\halfwidth,\heightplusside) -- (\width,\side) -- (\width,0) -- (\halfwidth,-\height) -- (0,0);
	
	\draw[hedge] (0,\side) -- (-\halfwidth,\heightplusside) -- (-\width,\side) -- (-\width,0) -- (-\halfwidth,-\height) -- (0,0);
	
	\draw[hedge] (-\halfwidth,-\height) -- (-\halfwidth,-\height-\side) -- (0,-\height-\side-\height) -- (\halfwidth,-\height-\side) -- (\halfwidth,-\height);
	
	\draw[hedge] (-\width, 0) -- (-1.5*\width, -\height) -- (-1.5*\width, -\height-\side) -- (-\width, -2*\height-\side) -- (-\halfwidth, -\height-\side);
	
	
	\draw[hedge] (-\width, -2*\height-\side) -- (-\width, -2*\height-2*\side) -- (-\halfwidth, -3*\height-2*\side) -- (0, -2*\height-2*\side) -- (0, -2*\height-\side);
	
	\draw[hedge] (\halfwidth, -\height-\side) -- (\width, -2*\height-\side) -- (\width, -2*\height-2*\side) -- (\halfwidth, -3*\height-2*\side) -- (0, -2*\height-2*\side);
	
	\node[smallvx, label=left:$a$] (a) at (-\halfwidth,-\height-0.5*\side) {};
	\node[smallvx, label=right:$b$] (b) at (\halfwidth,-\height-0.5*\side) {};

\end{scope}

\begin{scope}[shift={(\offset,0)}]

	\draw[hedge] (0,0) -- (0,\side) -- (\halfwidth,\heightplusside) -- (\width,\side) -- (\width,0) -- (\halfwidth,-\height) -- (0,0);
	\draw[hedge] (0,\side) -- (-\halfwidth,\heightplusside) -- (-\width,\side) -- (-\width,0) -- (-\halfwidth,-\height) -- (0,0);
	
	\begin{scope}[shift={(0,-3*\side)}]
		\draw[hedge] (0,0) -- (0,\side) -- (\halfwidth,\heightplusside) -- (\width,\side) -- (\width,0) -- (\halfwidth,-\height) -- (0,0);
		\draw[hedge] (0,\side) -- (-\halfwidth,\heightplusside) -- (-\width,\side) -- (-\width,0) -- (-\halfwidth,-\height) -- (0,0);
	\end{scope}
	
	\draw[hedge] (-\width,\side) -- (-1.5*\width, -0.25*\side) -- (-1.5*\width, -1.75*\side) -- (-\width,-3*\side);
	
	\node[smallvx, label=left:$a$] (a) at (-2*\width,-\side) {};
	\node[smallvx, label=right:$b$] (b) at (2*\width,-\side) {};
	
	\draw[hedge, out=90,in=90] (-\halfwidth, \side+\height) to (a);
	\draw[hedge, out=270,in=270] (-\halfwidth, -3*\side-\height) to (a);
	
	\draw[hedge, out=90,in=90] (\halfwidth, \side+\height) to (b);
	\draw[hedge, out=270,in=270] (\halfwidth, -3*\side-\height) to (b);
\end{scope}

\end{tikzpicture}
\caption{Left: A $B_6$ as we would like to embed it in a condensed wall, with positions of $a$ and $b$. Right: An isomorphic graph that reflects our acutal embedding of $B_6$ in $W^-$ as in Figure~\ref{fig:6BrickWallInCondWall}.}
\label{fig:6BrickWall}
\end{figure}

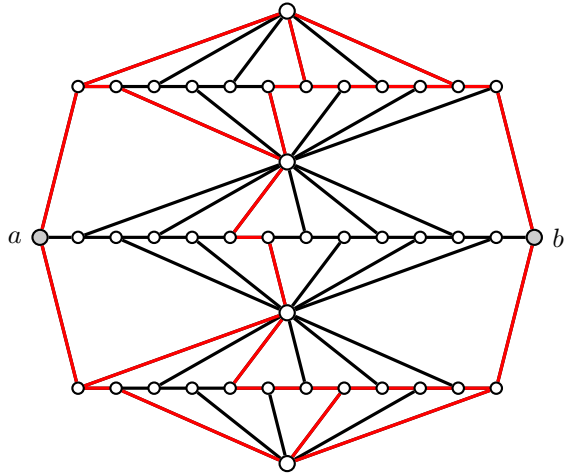
\begin{figure}[hbt] 
\centering
\begin{tikzpicture}[scale=1]
\tikzstyle{tinyvx}=[thick,circle,inner sep=0.cm, minimum size=1.5mm, fill=white, draw=black]

\def\vstep{2}
\def\hstep{0.5}
\def\hwidth{11}
\def\hheight{2}

\def\totalheight{\hheight*\vstep}
\def\totalwidth{\hwidth*\hstep}
\pgfmathtruncatemacro{\minustwo}{\hwidth-2}
\pgfmathtruncatemacro{\minusone}{\hwidth-1}

\foreach \j in {0,...,\hheight} {
	\draw[hedge] (0,-\j*\vstep) -- (\hwidth*\hstep,-\j*\vstep);
	\foreach \i in {0,...,\hwidth} {
		\node[tinyvx] (v\i\j) at (\i*\hstep,-\j*\vstep){};
	}
}
\pgfmathtruncatemacro{\plusvone}{\hheight+1}

\foreach \j in {0,...,\plusvone}{
	\node[hvertex] (z\j) at (0.5*\hwidth*\hstep,-\j*\vstep+0.5*\vstep) {};
}



\foreach \j in {1,...,\plusvone}{
	\pgfmathtruncatemacro{\subone}{\j-1}
	\foreach \i in {1,3,...,\hwidth}{
		\draw[hedge] (z\j) to (v\i\subone);
	}
}

\foreach \j in {0,...,\hheight}{
	\foreach \i in {0,2,...,\hwidth}{
		\draw[hedge] (z\j) to (v\i\j);
	}
}


\node[hvertex,fill=hellgrau,label=left:$a$] (a) at (-\hstep,-0.5*\totalheight) {};
\foreach \j in {0,...,\hheight} {
\draw[hedge] (a) to (v0\j);
}

\node[hvertex,fill=hellgrau,label=right:$b$] (b) at (\totalwidth+\hstep,-0.5*\totalheight) {};
\foreach \j in {0,...,\hheight} {
\draw[hedge] (v\hwidth\j) to (b);
}

\begin{pgfonlayer}{foreground}
	\draw[red, hedge] (z0) -- (v00) -- (v10) -- (z1) -- (v50) -- (v60) -- (z0);
	\draw[red, hedge] (z0) -- (v100) -- (v90) -- (v80) -- (v70) -- (v60);
	
	\draw[red, hedge] (z1) -- (v41) -- (v51) -- (z2);
	\draw[red, hedge] (v00) -- (a) -- (v02);
	\draw[red, hedge] (v100) -- (v110) -- (b) -- (v112);
	
	\draw[red, hedge] (z2) -- (v02) -- (v12) -- (z3) -- (v72) -- (v62) -- (v52) -- (v42) --(z2);
	\draw[red, hedge] (z3) -- (v112) -- (v102) -- (v92) -- (v82) -- (v72);

\end{pgfonlayer}

\end{tikzpicture}
\caption{A $B_6$ in three layers of a modified condensed wall $W^-$ (i.e. without jump-edges) of size~$6$. For more details on the embedding, see Figure~\ref{fig:6BrickWall}.}
\label{fig:6BrickWallInCondWall}
\end{figure}


\begin{lemma} \label{wall:lem:manyB7inCondensedWallPlusCDPath}
For every $n, r \in\N$, let $G$ be the graph consisting of a condensed wall without jump-edges $W^-$ of size $5 \cdot (n + r)$ and $n+r$ edge-disjoint paths $P$ (all internally disjoint from $W$) connecting $c$ with $d$. Then there are $n$~edge-disjoint $B_7$ in $G$, even with $r$ edges of $G$ being deleted.
\end{lemma}

\begin{proof}
Even with $r$ edges of $W^-$ being deleted, there are still at least $n$ untouched edge-disjoint chunks of $W^-$ that each contain five consecutive layers of $W^-$ and which all have an edge to both $a$ and $b$. In each of them, we can find a $B_7^-$ (that only misses one $z_i$-$z_{j+1}$-path) as in Figures~\ref{fig:7BrickWallEmbedding} and~\ref{fig:7BrickWallInCondWall}. 
Clearly, we can find edge-disjoint paths from each chunk to $c$ and $d$ in $W- \{a, b\}$, which together with a path $P$ outside of $W$ yields the desired $z_i$-$z_{j+1}$-path. Together, they form $n$ edge-disjoint $B_7$.
\end{proof}

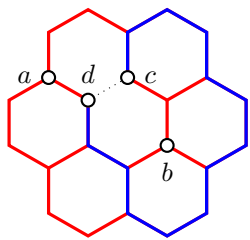
\begin{figure}[hbt] 
\centering
\begin{tikzpicture}[scale=0.6]
\tikzstyle{tinyvx}=[thick,circle,inner sep=0.cm, minimum size=1.5mm, fill=white, draw=black]

\def\side{1}
\def\height{0.5} 
\def\width{1.732} 
\def\halfwidth{0.5*\width}
\def\heightplusside{\height+\side}

\draw[red, hedge] (0,0) -- (0,\side) (\halfwidth,\heightplusside) -- (\width,\side) -- (\width,0) -- (\halfwidth,-\height) -- (0,0);
\draw[red, hedge] (0,\side) -- (-\halfwidth,\heightplusside) -- (-\width,\side) -- (-\width,0) -- (-\halfwidth,-\height) -- (0,0);
\draw[red, hedge] (-\halfwidth,\heightplusside) -- (-\halfwidth,\heightplusside+\side) -- (0,2*\height+2*\side) -- (\halfwidth,\heightplusside+\side) -- (\halfwidth,\heightplusside);
\draw[red, hedge] (-\halfwidth,-\height) -- (-\halfwidth,-\height-\side) -- (0,-2*\height-\side) -- (\halfwidth,-\height-\side) -- (\halfwidth,-\height);
\draw[red, hedge] (\halfwidth,-\height-\side) -- (\width,-2*\height-\side) -- (\width+\halfwidth,-\height-\side) -- (\width+\halfwidth,-\height) -- (\width,0);
\draw[red, hedge] (\width,\side) -- (\width+\halfwidth,\heightplusside) -- (\width+\halfwidth,\heightplusside+\side) -- (\width,2*\side+2*\height) -- (\halfwidth,\heightplusside+\side);
\draw[red, hedge] (\width+\halfwidth,-\height) -- (\width+2*\halfwidth,0) -- (\width+2*\halfwidth,\side) -- (\width+\halfwidth,\heightplusside);

\node[smallvx, label=above:$d$] (d) at (0,\side) {};
\node[smallvx, label=right:$c$] (c) at (\halfwidth,\heightplusside) {};
\node[smallvx, label=left:$a$] (a) at (-\halfwidth,\heightplusside) {};
\node[smallvx, label=below:$b$] (b) at (\width,0) {};

\draw[dotted] (c) -- (d);


\draw[blue, hedge] (c) -- (\halfwidth,\side+\heightplusside) -- (\width,2*\side+2*\height) -- (\width+\halfwidth,\side+\heightplusside) -- (\width+\halfwidth,\heightplusside) -- (2*\width,\side) -- (2*\width,0) -- (\width+\halfwidth,-\height) -- (\width+\halfwidth,-\height-\side) -- (\width,-2*\height-\side) -- (\halfwidth,-\height-\side) -- (\halfwidth,-\height) -- (0,0) -- (d) ;

\end{tikzpicture}
\caption{A $B_7$ as we embed it in Figure~\ref{fig:7BrickWallInCondWall}.}
\label{fig:7BrickWallEmbedding}
\end{figure}

\begin{figure}[hbt] 
\centering
\begin{tikzpicture}[scale=1]
\tikzstyle{tinyvx}=[thick,circle,inner sep=0.cm, minimum size=1.5mm, fill=white, draw=black]

\def\vstep{2}
\def\hstep{0.5}
\def\hwidth{11}
\def\hheight{4}

\def\totalheight{\hheight*\vstep}
\def\totalwidth{\hwidth*\hstep}
\pgfmathtruncatemacro{\minustwo}{\hwidth-2}
\pgfmathtruncatemacro{\minusone}{\hwidth-1}

\foreach \j in {0,...,\hheight} {
	\draw[hedge] (0,-\j*\vstep) -- (\hwidth*\hstep,-\j*\vstep);
	\foreach \i in {0,...,\hwidth} {
		\node[tinyvx] (v\i\j) at (\i*\hstep,-\j*\vstep){};
	}
}
\pgfmathtruncatemacro{\plusvone}{\hheight+1}

\foreach \j in {0,...,\plusvone}{
	\node[hvertex] (z\j) at (0.5*\hwidth*\hstep,-\j*\vstep+0.5*\vstep) {};
}



\foreach \j in {1,...,\plusvone}{
	\pgfmathtruncatemacro{\subone}{\j-1}
	\foreach \i in {1,3,...,\hwidth}{
		\draw[hedge] (z\j) to (v\i\subone);
	}
}

\foreach \j in {0,...,\hheight}{
	\foreach \i in {0,2,...,\hwidth}{
		\draw[hedge] (z\j) to (v\i\j);
	}
}


\node[hvertex,fill=hellgrau,label=left:$a$] (a) at (-\hstep,-0.5*\totalheight) {};
\foreach \j in {0,...,\hheight} {
\draw[hedge] (a) to (v0\j);
}

\node[hvertex,fill=hellgrau,label=right:$b$] (b) at (\totalwidth+\hstep,-0.5*\totalheight) {};
\foreach \j in {0,...,\hheight} {
\draw[hedge] (v\hwidth\j) to (b);
}

\begin{pgfonlayer}{foreground}
	\draw[blue, hedge] (z0) -- (v20) -- (v30) -- (v40) -- (v50) -- (z1) 
	 -- (v21) -- (v31) -- (v41) -- (v51) -- (z2)
	 -- (v42) -- (v52) -- (z3)
	 -- (v103) -- (v113) -- (z4)
	 -- (v44) -- (v54) -- (z5)  ; 
	
	\draw[red, hedge] (a) -- (v00) -- (v10) -- (v20);
	\draw[red, hedge] (z0) -- (v100) -- (v110) -- (b);
	\draw[red, hedge] (z1) -- (v110);
	
	\draw[red, hedge] (v51) -- (v61) -- (v71) -- (v81) -- (v91) -- (v101) -- (v111) -- (b);

	\draw[red, hedge] (a) -- (v03) -- (v13) -- (v23) -- (v33) -- (z4) ;
	\draw[red, hedge] (v33) -- (v43) -- (v53) -- (z3);
	\draw[red, hedge] (v113) -- (b);
	
	\draw[red, hedge] (a) -- (v04) -- (v14) -- (z5);

\end{pgfonlayer}

\end{tikzpicture}
\caption{A $B_7$ of which all but one edge fit into a condensed wall of size~$5$. A schematic drawing of what is embedded here can be found in Figure~\ref{fig:7BrickWallEmbedding}.}
\label{fig:7BrickWallInCondWall}
\end{figure}
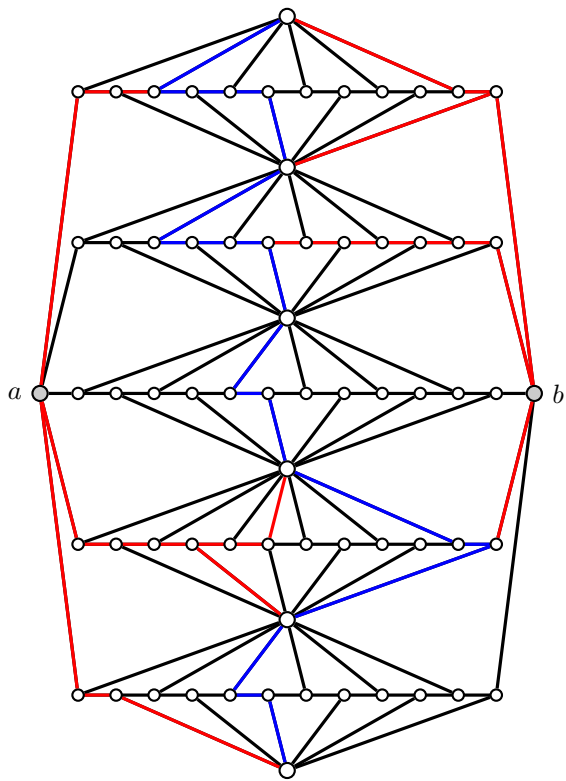

\begin{prop} \label{wall:prop:condWallB7noProofPossible}
A condensed wall $W$ (or its modified counterpart $W^-$) with some connections of its terminals cannot serve as a counterexample to prove that a $B_7$ does not have the edge-Erd\H{o}s-P\'{o}sa property.
\end{prop}

\begin{proof}
Let $G$ be our counterexample graph, i. e. a condensed wall $W$ (or its modified counterpart $W^-$) with some connections of its terminals. Let $U$ be a $B_7$ in $G$. Let $U_{out}$ be the subgraph of $U$ in $G - (W - \{a, b\})$, while $U_{in}$ is the subgraph of $U$ in $G$.
In Lemma~\ref{wall:lem:noB7inCondensedWall}, we have seen that a $B_7$ cannot be found in any condensed wall $W$. As a $B_7$ is $2$-connected, we thus need at least two terminals of $W$ to be connected outside of $W$. However, Lemma~\ref{wall:lem:manyB7inCondensedWallPlusCDPath} shows that we must not connect $c$ with $d$, as we would otherwise find arbitrarily many $B_7$. This limits us to at most three connections, which in turn implies that $U_{out}$ is connected.

If $U_{out}$ has only $2$~connections to $W$, then $U_{out}$ only contains a single path with at least one of $a$ and $b$ as an endvertex. We can again conclude with Lemma~\ref{wall:lem:noB7inCondensedWall} that we cannot find a $B_7$ in $G$.
It remains to check what happens if $U_{out}$ uses three connections, namely $a, b$ and exactly one of $c$ and $d$. Then $U_{out}$ contains a single vertex $v$ of degree~$3$ in $B_7$ and possibly some paths incident with it. This implies that there is no vertex of degree~$3$ in $B_7$ between $a, b$ and $v$. Therefore, $U_{out}$ is only incident with at most $4$ of the outer bricks of $B_7$, and those are adjacent to each other. This implies that there is a $B_2$ in $W - \{a, b\}$, a contradiction to Lemma~\ref{wall:lem:atMost3ConnectionsForB2}.
\end{proof}

\subsection{Main Proof}

Let us now come back to proving our main lemma.

\begin{lemma*linkageWalls}
$U$ contains an ($a$--$b$, $c$--$d$) linkage in $W$.
\end{lemma*linkageWalls}

\begin{proof}
$U$ can not be entirely contained in $W$ due to Lemma~\ref{wall:lem:noB7inCondensedWall}. Cleary, it can not be entirely contained in $G^* - (W - \{a, b, c, d\})$ (the graph which contains all edges of $G^*$ that are not in $W$) either, as there are two edges missing there. Therefore, $U$ must have edges in both $W$ and $G^* - (W - \{a, b, c, d\})$. This motivates the following definitions:

Let $U_{in}$ be the subgraph of $U$ in $W$, and let $U_{out}$ be the subgraph of $U$ in $G^* - (W - \{a, b, c, d\})$. Note that $U_{in}$ and $U_{out}$ are edge-disjoint, $U_{in} + U_{out} = U$ and the only vertices that can be shared by $U_{in}$ and $U_{out}$ are $a, b, c$ and $d$. Therefore, we define a \emph{connecting vertex} as a vertex of $U$ that is incident with edges of both $U_{in}$ and $U_{out}$.

\vspace{1ex}
\noindent\textbf{Case~1}: Every component of $U_{in}$ contains at most three connecting vertices.
\vspace{1ex}

As $U$ is $2$-connected, every component of $U_{in}$ must contain at least two connecting vertices. Of the body of $U$, one of $U_{in}$ or $U_{out}$ contains all but at most one vertex of degree~$3$ (the \emph{large component}), while the other may only contain either one or two subgraphs of maximum degree~$2$ (i.e. paths) or a single subgraph with one single vertex of degree~$3$ and possibly some incident paths.

Suppose the large component is in $W$. Since $B$ is a wall of size at least $6 \times 4$, the body of $B$ contains a $B_{10}$. (Actually, it is considerably larger, but a $B_{10}$ will suffice to obtain a contradiction in this case.)
If $U_{out}$ only contains paths, we know by Lemma~\ref{wall:lem:noB7inCondensedWall} that if every one of those paths has $a$ or $b$ as one endvertex, we would not be able to embed a $B_7$ in $W \cup U_{out}$, lest there cannot be a $B_{10}$ or the entire body of $U$ in $G^*$. Therefore, we conclude that $U_{out}$ must either contain a vertex of degree~$3$ or a $cd$-path (and possibly also an $ab$-path).

Next, we note that the eight outer bricks of the $B_{10}$ 
are arranged in a cycle, which we call \emph{grand cycle}. This grand cycle contains two disjoint cycles that are incident with all eight bricks of the grand cycle, which we call \emph{large cycles}. $U_{in}$ still contains most of this grand cycle, except for one or two small subgraphs as described above.
Each vertex can be part of at most two bricks in the grand cycle. Additionally, it can be part of at most one of the large cycles. 

If there is a vertex $v$ of degree~$3$ or a $cd$-path $P$ in $U_{out}$, $W - \{a, b\}$ still contains at least all but three vertices of degree~$3$ in the grand cycle and all the paths between them that are not incident with a missing vertex. This implies $W - \{a, b\}$ contains at least two bricks. If those are adjacent, there is a $B_2$ in $W - \{a, b\}$, a contradiction to Lemma~\ref{wall:lem:atMost3ConnectionsForB2}. If every two bricks in $W - \{a, b\}$ are not adjacent, we have a look at how many bricks in $W - \{a, b\}$ are left. If there are exactly two, each of $a, b$ and $v$ (or $P$) must be incident with two bricks, and those bricks are pairwise different for each vertex. We conclude that we can still find a cycle in $W - \{a, b\}$ that is incident with all bricks of the grand cycle. Together with the two bricks in $W - \{a, b\}$, this yields a $B_1^2$ in $W - \{a, b\}$. As a $B_1^2$ is $2$-connected, it is contained in a single layer of $W$. However, it has two disjoint connections to each of $a, b$ and $v$ (or $P$) in $U$, which means there must be six pairwise disjoint paths connecting $B_1^2$ with vertices outside of its layer, which is clearly not possible.

What happens if there are more than two bricks left in $W - \{a, b\}$? Then, as no two of them are adjacent, there are exactly three bricks in $W - \{a, b\}$ and exactly one of $a, b$ and $v$ (or $P$) lies between every two of them. Clearly, we can again find a $B_1^2$ in $W - \{a, b\}$ and arrive at a contradiction as we did above.


It remains to check what happens if the large component $C$ is contained in $U_{out}$. We claim that if there is an embedding of $C$ in $G^* - (W - \{a, b, c, d\})$, then there is also an embedding of $C$ in the body of $G^* - (W - \{a, b, c, d\})$. We observe that no brick of $C$ can be entirely contained in an arm of $G^* - (W - \{a, b, c, d\})$ by definition of the body. If there is a brick $B_1$ of $C$ that is only partly contained in an arm of $G^* - (W - \{a, b, c, d\})$, then the other part of $B_1$ is contained in the body of $G^* - (W - \{a, b, c, d\})$. In particular, there is a path $P$ (the border between body and arm) such that $B_1 \cup P$ is a $B_2$. But of this $B_2$ exactly one brick is entirely contained in the arm while the other is entirely contained in the body of $G^* - (W - \{a, b, c, d\})$. Replacing $B_1$ with the latter and doing the same for all bricks of $C$ that are not entirely contained in $G^* - (W - \{a, b, c, d\})$ yields an embedding of $C$ in $G^* - (W - \{a, b, c, d\})$.

The body of $G^* - (W - \{a, b, c, d\})$ is missing two edges ($e_1$ and $e_2$), so the body of $U$ cannot be entirely contained there. To see that connecting $a$ or $b$ with $c$ or $d$ would not help to embed the body of $U$, we count bricks. $G^*$ misses $e_1, e_2$, so it contains two bricks less than $B$. A cycle that uses exactly one new $\{a,b\}-\{c,d\}$~path would be incident with at least seven bricks, meaning the cycle can not constitute a new brick since every brick can only be incident with at most 6 other bricks. Now let $K$ be a cycle that uses both new $\{a,b\}-\{c,d\}$~paths and is incident with at most $6$~bricks. Then it is incident with all three bricks that are adjacent to the brick containing $e_1$, and also incident with the three bricks adjacent to the brick containing $e_2$. But those two groups of three bricks each are not adjacent to each other, which is impossible for the neighbourhood of a brick. We conclude that $K$ may not constitute a new brick, either. We conclude that connecting $a$ or $b$ with $c$ or $d$ does not allow for an embedding the body of $U$.

Therefore, the only possibility is to add an $a$-$b$~path and a $c$-$d$~path in $W$. Those paths must be disjoint, as otherwise the two bricks repaired by using them would be adjacent, which is impossible since their neighbourhood is not. Thus $U$ contains an $a-b,c-d$-linkage, which was what we wanted.

\vspace{1ex}
\noindent\textbf{Case~2}: There are exactly four connecting vertices in one component of $U_{in}$.
\vspace{1ex}

By definition, there can be at most four connecting vertices in total, so this case covers everything not handled in Case~1. In particular, both $U_{in}$ and $U_{out}$ are connected. 

If $U_{in}$ does not contain a single brick, it contains no cycle and is therefore a tree with exactly four leaves and exactly two vertices of degree~$3$. If $U_{in}$ contains an $a-b$,$c-d$-linkage, there is nothing left to show. Otherwise, it looks as the red part in Figure~\ref{wall:fig:treeInUin}.

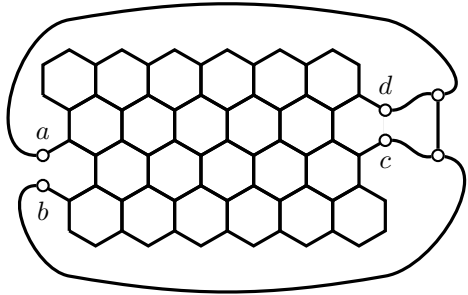
\begin{figure}[hbt]  \label{wall:fig:treeInUin}
\centering
\begin{tikzpicture}[scale=0.4]
\tikzstyle{tinyvx}=[thick,circle,inner sep=0.cm, minimum size=1.5mm, fill=white, draw=black]

\def\hbricks{6}

\def\side{1}
\def\height{0.5} 
\def\width{1.732} 
\def\halfwidth{0.5*\width}
\def\heightplusside{\height+\side}

\foreach \i in {1, ..., \hbricks} {
	\foreach \j in {0, 3} {
		\draw[hedge] (\i*\width,\j) -- (\i*\width,\side + \j) -- (\halfwidth + \i*\width,\heightplusside + \j) -- (\width + \i*\width,\side + \j) -- (\width + \i*\width,\j) -- (\halfwidth + \i*\width,-\height + \j) -- (\i*\width,\j);
	}
}

\begin{scope}[shift={(-\halfwidth,1.5)}]
\foreach \i in {1, ..., \hbricks} {
	\foreach \j in {0, 3} {
		\draw[hedge] (\i*\width,\j) -- (\i*\width,\side + \j) -- (\halfwidth + \i*\width,\heightplusside + \j) -- (\width + \i*\width,\side + \j) -- (\width + \i*\width,\j) -- (\halfwidth + \i*\width,-\height + \j) -- (\i*\width,\j);
	}
}
\end{scope}

\node[tinyvx,label=above:$a$] (a) at (\halfwidth,2.5) {};
\node[tinyvx,label=below:$b$] (b) at (\halfwidth,1.5) {};

\node[tinyvx,label=above:$d$] (d) at (7*\width, 4) {};
\node[tinyvx,label=below:$c$] (c) at (7*\width,3) {};

\draw[ultra thick, white] (a) -- (b);
\draw[ultra thick, white] (c) -- (d);

\node[tinyvx, draw=bold] (v1) at (8*\width, 4.5){};
\node[tinyvx, draw=bold] (v2) at (8*\width, 2.5){};

\draw[rededge] (a) to [out=180,in=190] (\width, 7) to [out=10,in=170] (7*\width, 7) to [out=350,in=10] (v1) to [out=170,in=0] (d);
\draw[rededge] (b) to [out=180,in=170] (\width,-1.5) to [out=350,in=190] (7*\width,-1.5) to [out=10,in=350] (v2) to [out=190,in=0] (c);
\draw[rededge] (v1) -- (v2);
\end{tikzpicture}
\caption{$G^* - W^0 + U_{in}$ when $U_{in}$ contains neither a brick nor a linkage.}
\end{figure}

To show that $G^* - W^0 + U_{in}$ does not contain an embedding of $B$ in this case (see Figure~\ref{wall:fig:treeInUin}), we count the bricks in the body of $B$ again. Since $G^*$ is missing $e_1$ and $e_2$, its body is missing two bricks. We have seen before that no cycle that contains exactly one $\{a,b\}$-$\{c,d\}$~path may constitute a new brick since it it incident with at least seven other bricks.

$G^* - W^0 + U_{in}$ may contain exactly one new brick that contains an $a$-$b$~path or a $c$-$d$-path in $U_{in}$. However, if there were two such bricks, they were adjacent to each other, which is impossible since their other neighbours are those bricks that were adjadent to the bricks containing $e_1$ and $e_2$, respectively. This would imply that each of the two new bricks' neighbourhood would be disconnected, which is impossible. We conclude that the body of $G^* - W^0 + U_{in}$ contains one brick less than the body of $B$, a contradiction.


As $U_{in}$ contains at least one brick, there is a cycle $C$ in $U_{in}$ that contains all bricks of $U_{in}$ on its inner face. In particular, the four paths connecting $a, b, c$ and $d$ all end in $C$ in pairwise different vertices $v_a, v_b, v_c$ and $v_d$. If there is a $v_a v_b$-path in $C$ disjoint from $v_c$ and $v_d$, then there is also a (disjoint) $v_c v_d$-path in $C$ disjoint from $v_a$ and $v_b$. Together, they yield an $a-b$,$c-d$-linkage in $W$. We may therefore assume that every $v_a v_b$-path on $C$ is incident with $v_c$ or $v_d$. This implies that out of the four paths connecting $U_{in}$ and $U_{out}$, the two containing $a$ and $b$ do not lie next to each other, but instead alternate with the ones containing $c$ and $d$.

Now we can study $U_{out}$. As $U_{in}$ cannot contain a $B_7$, there are some bricks left in $U_{out}$ and four paths connecting it to $U_{in}$. The border to the outer face of the bricks in $U_{out}$ forms a cycle $C$ that is incident with all paths that connect $U_{in}$ with $U_{out}$. Therefore, $C$ must be incident with $a, b, c$ and $d$ in $U_{out}$. We have seen that the connecting paths containing $a$ and $b$ alternate with those containing $c$ and $d$. Therefore, there must be an $ac$-path and a disjoint $bd$ path in $U_{out}$. This is impossible by construction.
\end{proof}

\section{Discussion}

In Theorem~\ref{wall:thm:WallsHaveNotEdgeErdosPosaIf6x4Contained}, we have seen that the expansions of every wall of size at least $6 \times 4$ do not have the edge-Erd\H{o}s-P\'{o}sa property. Is $6 \times 4$ optimal?

When studying the proof of Lemma~\ref{wall:lem:linkageMainLemma}, one will notice that a size of $6 \times 4$ was only needed to get an easier argument for why adding an $\{a,b\}$-$\{c,d\}$~path (that is, a path connecting vertices that are far away from each other) does not help to create new bricks. However, I believe that with a more careful argument (for example counting the neighbours of every brick), the same result can be obtained for smaller walls such as a $B_{10}$. In particular, you may have noticed that those parts of the proof of Lemma~\ref{wall:lem:linkageMainLemma} that referred to embedding a large part of a wall into the Heinlein Wall already made do with a $B_{10}$. Can we do even better?

In Proposition~\ref{wall:prop:condWallB7noProofPossible}, we showed that a condensed wall cannot be used to prove that a $B_7$ does not have the edge-Erd\H{o}s-P\'{o}sa property, even though it is not a minor of it as seen in Lemma~\ref{wall:lem:noB7inCondensedWall}. What about a $B_8$ or $B_9$? I believe that a Heinlein Wall cannot be used to prove that they do not have the edge-Erd\H{o}s-P\'{o}sa property, either. I did not prove that, but I found evidence in the form of constructions that place most bricks of a $B_8$ or $B_9$ in a (modified) Heinlein Wall $W^-$ while needing only one or two bricks outside of it with three or four connections to $W^-$.

\begin{lemma} \label{wall:lem:manyB8inCondensedWallPlusB1}
For every $n, r \in\N$, let $G$ be the graph consisting of a Heinlein Wall without jump-edges $W^-$ of size $3 \cdot (n + r)$ and a $n+r$-fold brick $B_1$ with three adjacent branch vertices being connected to $d, b$ and $c$ (in this order when counting clockwise) by $n+r$-fold paths (all internally disjoint from $W$, $B_1$ and each other). Then there are $n$~edge-disjoint $B_8$ in $G$, even with $r$ edges of $G$ being deleted.
\end{lemma}

\begin{proof}
Even with $r$ edges of $G$ being deleted, there are still at least $n$ untouched edge-disjoint chunks of $W^-$ that each contain three consecutive layers of $W^-$ and which all have an edge to both $a$ and $b$. In each of them, we can find a $B_8^-$ (that only misses one brick and three connecting paths) as in Figures~\ref{fig:8BrickWallEmbedding} and~\ref{fig:8BrickWallInCondWall}.
Clearly, we can find edge-disjoint paths from each chunk to $c$ and $d$ in $W^- - \{a, b\}$, which together with a brick and three connections outside of $W^-$ yield the desired $B_8$. Together, they form $n$ edge-disjoint $B_8$.
\end{proof}

\begin{figure}[hbt] 
\centering
\begin{tikzpicture}[scale=0.6]
\tikzstyle{tinyvx}=[thick,circle,inner sep=0.cm, minimum size=1.5mm, fill=white, draw=black]

\def\side{1}
\def\height{0.5} 
\def\width{1.732} 
\def\halfwidth{0.5*\width}
\def\heightplusside{\height+\side}

\draw[red, hedge] (0,0) -- (0,\side) -- (\halfwidth,\heightplusside) -- (\width,\side) -- (\width,0) -- (\halfwidth,-\height) -- (0,0);
\draw[red, hedge] (0,\side) -- (-\halfwidth,\heightplusside) -- (-\width,\side) -- (-\width,0) -- (-\halfwidth,-\height) -- (0,0);
\draw[red, hedge] (-\halfwidth,\heightplusside) -- (-\halfwidth,\heightplusside+\side) -- (0,2*\height+2*\side) -- (\halfwidth,\heightplusside+\side) -- (\halfwidth,\heightplusside);
\draw[red, hedge] (-\halfwidth,-\height) -- (-\halfwidth,-\height-\side) -- (0,-2*\height-\side) -- (\halfwidth,-\height-\side) -- (\halfwidth,-\height);
\draw[red, hedge] (\halfwidth,-\height-\side) -- (\width,-2*\height-\side) -- (\width+\halfwidth,-\height-\side) -- (\width+\halfwidth,-\height) -- (\width,0);
\draw[dotted] (\width,\side) -- (\width+\halfwidth,\heightplusside) -- (\width+\halfwidth,\heightplusside+\side) -- (\width,2*\side+2*\height) -- (\halfwidth,\heightplusside+\side);
\draw[dotted] (\width+\halfwidth,-\height) -- (\width+2*\halfwidth,0) -- (\width+2*\halfwidth,\side) -- (\width+\halfwidth,\heightplusside);
\draw[dotted] (2*\width,\side) -- (2*\width+\halfwidth,\heightplusside) -- (2*\width+\halfwidth,\heightplusside+\side) -- (2*\width,2*\side+2*\height) -- (\width+\halfwidth,\heightplusside+\side);

\node[smallvx, label=above:$c$] (c) at (\halfwidth,2*\side+\height) {};
\node[smallvx, label=right:$b$] (b) at (\width,\side) {};
\node[smallvx, label=right:$d$] (d) at (\width+\halfwidth,-\height) {};
\node[smallvx, label=left:$a$] (a) at (-\halfwidth,-\height) {};


\draw[blue, hedge] (c) -- (\halfwidth,\heightplusside) -- (0,\side) -- (0,0) -- (\halfwidth,-\height) -- (\width,0) -- (d) ;

\end{tikzpicture}
\caption{A $B_8$ as we embed it in Figure~\ref{fig:8BrickWallInCondWall}.}
\label{fig:8BrickWallEmbedding}
\end{figure}
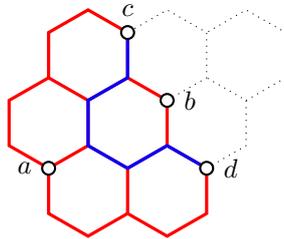

\begin{figure}[hbt] 
\centering
\begin{tikzpicture}[scale=1]
\tikzstyle{tinyvx}=[thick,circle,inner sep=0.cm, minimum size=1.5mm, fill=white, draw=black]

\def\hoffset{8}
\def\voffset{-2}

\def\vstep{2}
\def\hstep{0.5}
\def\hwidth{11}
\def\hheight{2}

\def\totalheight{\hheight*\vstep}
\def\totalwidth{\hwidth*\hstep}
\pgfmathtruncatemacro{\minustwo}{\hwidth-2}
\pgfmathtruncatemacro{\minusone}{\hwidth-1}

\foreach \j in {0,...,\hheight} {
	\draw[hedge] (0,-\j*\vstep) -- (\hwidth*\hstep,-\j*\vstep);
	\foreach \i in {0,...,\hwidth} {
		\node[tinyvx] (v\i\j) at (\i*\hstep,-\j*\vstep){};
	}
}
\pgfmathtruncatemacro{\plusvone}{\hheight+1}

\foreach \j in {0,...,\plusvone}{
	\node[hvertex] (z\j) at (0.5*\hwidth*\hstep,-\j*\vstep+0.5*\vstep) {};
}



\foreach \j in {1,...,\plusvone}{
	\pgfmathtruncatemacro{\subone}{\j-1}
	\foreach \i in {1,3,...,\hwidth}{
		\draw[hedge] (z\j) to (v\i\subone);
	}
}

\foreach \j in {0,...,\hheight}{
	\foreach \i in {0,2,...,\hwidth}{
		\draw[hedge] (z\j) to (v\i\j);
	}
}


\node[hvertex,fill=hellgrau,label=left:$a$] (a) at (-\hstep,-0.5*\totalheight) {};
\foreach \j in {0,...,\hheight} {
\draw[hedge] (a) to (v0\j);
}

\node[hvertex,fill=hellgrau,label=right:$b$] (b) at (\totalwidth+\hstep,-0.5*\totalheight) {};
\foreach \j in {0,...,\hheight} {
\draw[hedge] (v\hwidth\j) to (b);
}

\begin{pgfonlayer}{foreground}
	\draw[blue, hedge] (z0) -- (v100) -- (v110) -- (z1) 
	 -- (v21) -- (v31) -- (z2)
	 -- (v102) -- (v112) -- (z3); 
	
	\draw[red, hedge] (a) -- (v00) -- (v10) -- (v20) -- (v30);
	\draw[red, hedge] (z0) -- (v60) -- (v50) -- (v40) -- (v30);
	\draw[red, hedge] (z1) -- (v30);
	
	\draw[red, hedge] (v110) -- (b) -- (v112);

	\draw[red, hedge] (a) -- (v02) -- (v12) -- (v22);
	\draw[red, hedge] (v22) -- (z2);
	\draw[red, hedge] (v22) -- (v32) -- (v42) -- (v52) -- (z3);
	
\end{pgfonlayer}


\def\side{1}
\def\height{0.5} 
\def\width{1.732} 
\def\halfwidth{0.5*\width}
\def\heightplusside{\height+\side}

\begin{scope}[shift={(\hoffset,\voffset)}]
	\draw[dotted] (0,0) -- (0,\side) -- (\halfwidth,\heightplusside) -- (\width,\side) -- (\width,0) -- (\halfwidth,-\height) -- (0,0);
	
	\node (x) at (0,\side) {};
	\node (y) at (0,0) {};
	\node (z) at (\halfwidth,-\height) {};
\end{scope}


\draw[dotted] (z0) -- (\totalwidth,1) -- (x);
\draw[dotted] (b) -- (y);
\draw[dotted] (z3) -- (\totalwidth,-\totalheight-1) -- (z);

\end{tikzpicture}
\caption{A $B_8$ embedded in three layers of a Heinlein Wall. A schematic drawing of what is embedded here can be found in Figure~\ref{fig:8BrickWallEmbedding}.}
\label{fig:8BrickWallInCondWall}
\end{figure}

\begin{lemma} \label{wall:lem:manyB9inCondensedWallPlusB2}
For every $n, r \in\N$, let $G$ be the graph consisting of a Heinlein Wall without jump-edges $W^-$ of size $3 \cdot (n + r)$ and a $n+r$-fold $B_2$ with four branch vertices being connected to $a, b, c$ and $d$ by $n-r$-fold paths (all internally disjoint from $W$, $B_2$ and each other and such that a $B_9$ can be formed, see Figure~\ref{fig:9BrickWallInCondWall}). Then there are $n$~edge-disjoint $B_9$ in $G$, even with $r$ edges of $G$ being deleted.
\end{lemma}

\begin{proof}
Even with $r$ edges of $G$ being deleted, there are still at least $n$ untouched edge-disjoint chunks of $W^-$ that each contain three consecutive layers of $W^-$ and which all have an edge to both $a$ and $b$. In each of them, we can find a $B_9^-$ (that only misses two bricks and four connecting paths) as in Figures~\ref{fig:9BrickWallEmbedding} and~\ref{fig:9BrickWallInCondWall}.
Clearly, we can find edge-disjoint paths from each chunk to $c$ and $d$ in $W^- - \{a, b\}$, which together with two bricks and four connections outside of $W^-$ yield the desired $B_9$. Together, they form $n$ edge-disjoint $B_9$.
\end{proof}

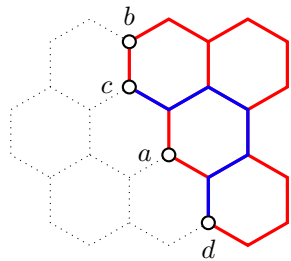
\begin{figure}[hbt] 
\centering
\begin{tikzpicture}[scale=0.6]
\tikzstyle{tinyvx}=[thick,circle,inner sep=0.cm, minimum size=1.5mm, fill=white, draw=black]

\def\side{1}
\def\height{0.5} 
\def\width{1.732} 
\def\halfwidth{0.5*\width}
\def\heightplusside{\height+\side}

\draw[dotted] (0,0) -- (0,\side) -- (\halfwidth,\heightplusside);
\draw[red, hedge] (\halfwidth,\heightplusside) -- (\width,\side) -- (\width,0);
\draw[dotted] (\width,0) -- (\halfwidth,-\height) -- (0,0);

\draw[dotted] (0,\side) -- (-\halfwidth,\heightplusside) -- (-\width,\side) -- (-\width,0) -- (-\halfwidth,-\height) -- (0,0);

\draw[dotted] (-\halfwidth,\heightplusside) -- (-\halfwidth,\heightplusside+\side) -- (0,2*\height+2*\side) -- (\halfwidth,\heightplusside+\side);
\draw[red, hedge] (\halfwidth,\heightplusside+\side) -- (\halfwidth,\heightplusside);

\draw[dotted] (-\halfwidth,-\height) -- (-\halfwidth,-\height-\side) -- (0,-2*\height-\side) -- (\halfwidth,-\height-\side) -- (\halfwidth,-\height);

\draw[dotted] (\halfwidth,-\height-\side) -- (\width,-2*\height-\side) -- (\width+\halfwidth,-\height-\side);
\draw[red, hedge] (\width+\halfwidth,-\height-\side) -- (\width+\halfwidth,-\height) -- (\width,0);

\draw[red, hedge] (\width,\side) -- (\width+\halfwidth,\heightplusside) -- (\width+\halfwidth,\heightplusside+\side) -- (\width,2*\side+2*\height) -- (\halfwidth,\heightplusside+\side);
\draw[red, hedge] (\width+\halfwidth,-\height) -- (\width+2*\halfwidth,0) -- (\width+2*\halfwidth,\side) -- (\width+\halfwidth,\heightplusside);
\draw[red, hedge] (\width+\halfwidth,-\height-\side) -- (2*\width,-2*\height-\side) -- (2*\width+\halfwidth,-\height-\side) -- (2*\width+\halfwidth,-\height) -- (2*\width,0);
\draw[red, hedge] (2*\width,\side) -- (2*\width+\halfwidth,\heightplusside) -- (2*\width+\halfwidth,\heightplusside+\side) -- (2*\width,2*\side+2*\height) -- (\width+\halfwidth,\heightplusside+\side);

\node[smallvx, label=left:$c$] (c) at (\halfwidth,\heightplusside) {};
\node[smallvx, label=above:$b$] (b) at (\halfwidth,\heightplusside+\side) {};
\node[smallvx, label=below:$d$] (d) at (\width+\halfwidth,-\height-\side) {};
\node[smallvx, label=left:$a$] (a) at (\width,0) {};


\draw[blue, hedge] (c) -- (\width,\side) -- (\width+\halfwidth,\heightplusside) -- (2*\width,\side) -- (2*\width,0) -- (\width+\halfwidth,-\height) -- (d) ;

\end{tikzpicture}
\caption{A $B_9$ as we embed it in Figure~\ref{fig:9BrickWallInCondWall}.}
\label{fig:9BrickWallEmbedding}
\end{figure}

\begin{figure}[hbt] 
\centering
\begin{tikzpicture}[scale=1]
\tikzstyle{tinyvx}=[thick,circle,inner sep=0.cm, minimum size=1.5mm, fill=white, draw=black]

\def\hoffset{-4}
\def\voffset{-2}

\def\vstep{2}
\def\hstep{0.5}
\def\hwidth{11}
\def\hheight{2}

\def\totalheight{\hheight*\vstep}
\def\totalwidth{\hwidth*\hstep}
\pgfmathtruncatemacro{\minustwo}{\hwidth-2}
\pgfmathtruncatemacro{\minusone}{\hwidth-1}

\foreach \j in {0,...,\hheight} {
	\draw[hedge] (0,-\j*\vstep) -- (\hwidth*\hstep,-\j*\vstep);
	\foreach \i in {0,...,\hwidth} {
		\node[tinyvx] (v\i\j) at (\i*\hstep,-\j*\vstep){};
	}
}
\pgfmathtruncatemacro{\plusvone}{\hheight+1}

\foreach \j in {0,...,\plusvone}{
	\node[hvertex] (z\j) at (0.5*\hwidth*\hstep,-\j*\vstep+0.5*\vstep) {};
}



\foreach \j in {1,...,\plusvone}{
	\pgfmathtruncatemacro{\subone}{\j-1}
	\foreach \i in {1,3,...,\hwidth}{
		\draw[hedge] (z\j) to (v\i\subone);
	}
}

\foreach \j in {0,...,\hheight}{
	\foreach \i in {0,2,...,\hwidth}{
		\draw[hedge] (z\j) to (v\i\j);
	}
}


\node[hvertex,fill=hellgrau,label=left:$a$] (a) at (-\hstep,-0.5*\totalheight) {};
\foreach \j in {0,...,\hheight} {
\draw[hedge] (a) to (v0\j);
}

\node[hvertex,fill=hellgrau,label=right:$b$] (b) at (\totalwidth+\hstep,-0.5*\totalheight) {};
\foreach \j in {0,...,\hheight} {
\draw[hedge] (v\hwidth\j) to (b);
}

\begin{pgfonlayer}{foreground}
	\draw[blue, hedge] (z0) -- (v00) -- (v10) -- (z1) 
	 -- (v61) -- (v71) -- (z2)
	 -- (v02) -- (v12) -- (z3); 
	
	\draw[red, hedge] (a) -- (v00);
	\draw[red, hedge] (a) -- (v02);
	
	\draw[red, hedge] (z0) -- (v100) -- (v110) -- (b);
	
	\draw[red, hedge] (z1) -- (v101) -- (v111) -- (b);
	\draw[red, hedge] (v71) -- (v81) -- (v91) -- (v101);
	
	\draw[red, hedge] (z2) -- (v42) -- (v52) -- (z3);
	
\end{pgfonlayer}


\def\side{1}
\def\height{0.5} 
\def\width{1.732} 
\def\halfwidth{0.5*\width}
\def\heightplusside{\height+\side}

\begin{scope}[shift={(\hoffset,\voffset)}]
	\draw[dotted] (0,0) -- (0,\side) -- (-\halfwidth,\heightplusside) -- (-\width,\side) -- (-\width,0) -- (-\halfwidth,-\height) -- (0,0);
	\draw[dotted] (-\halfwidth,-\height) -- (-\halfwidth,-\height-\side) -- (0,-\side-2*\height) -- (\halfwidth,-\height-\side) -- (\halfwidth,-\height) -- (0,0);
	
	\node (a1) at (\halfwidth,-\height) {};
	\node (b1) at (-\halfwidth,\heightplusside) {};
	\node (c1) at (0,\side) {};
	\node (d1) at (\halfwidth,-\height-\side) {};
\end{scope}


\draw[dotted] (b) -- (\totalwidth+1,2) -- (-0.5*\totalwidth,2) -- (b1);
\draw[dotted] (z0) -- (-0.5*\totalwidth,1) -- (c1);
\draw[dotted] (a) -- (a1);
\draw[dotted] (z3) -- (-0.5*\totalwidth,-\totalheight-1) -- (d1);

\end{tikzpicture}
\caption{A $B_9$ embedded in three layers of a condensed wall. A schematic drawing of what is embedded here can be found in Figure~\ref{fig:9BrickWallEmbedding}.}
\label{fig:9BrickWallInCondWall}
\end{figure}
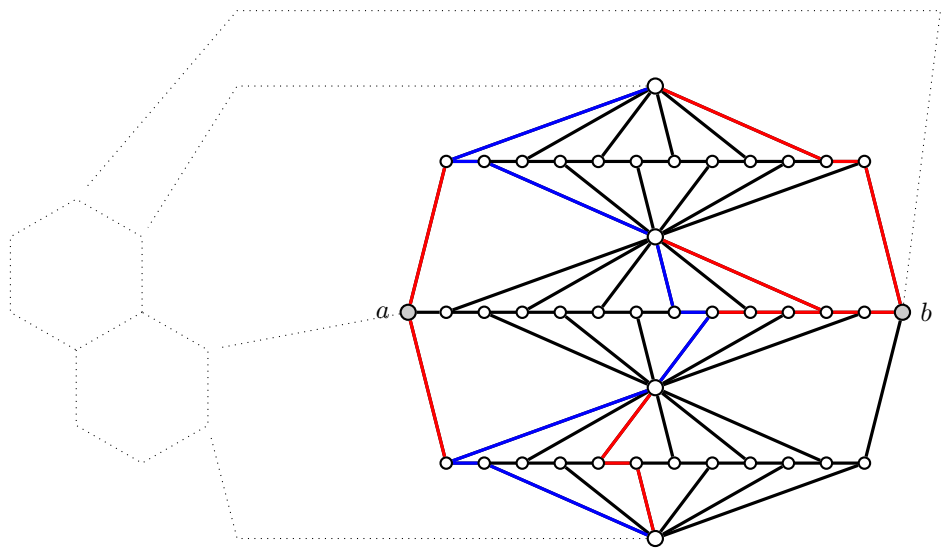

\bibliographystyle{amsplain}

\bibliography{chapters/literature}{}

%
%
%
%
%

\end{document}